\documentclass[10pt,a4paper,reqno]{amsart}
\usepackage[T1]{fontenc}
\usepackage[utf8]{inputenc}

\usepackage{amsthm,amsmath,amssymb,epsfig,graphicx}
\usepackage{color}
\usepackage{todonotes,csquotes,microtype}
\usepackage{hyperref}

\theoremstyle{plain}
\newtheorem{theorem}{Theorem}[section]
\newtheorem{lemma}[theorem]{Lemma}
\newtheorem{prop}[theorem]{Proposition}

\theoremstyle{remark}
{

\newtheorem{rem}[theorem]{Remark}

}

\newcommand{\eps}{\epsilon}

\newcommand{\ma}{\alpha}
\newcommand{\mb}{\beta}
\newcommand{\mg}{\gamma}
\newcommand{\ml}{\lambda}
\newcommand{\ms}{\sigma}
\newcommand{\mt}{\theta}
\newcommand{\mO}{\Omega}

\newcommand{\cA}{{\mathcal A}}
\newcommand{\cF}{{\mathcal F}}
\newcommand{\cH}{{\mathcal H}}
\newcommand{\cI}{{\mathcal I}}
\newcommand{\cS}{{\mathcal S}}
\newcommand{\cT}{{\mathcal T}}

\newcommand{\E} {{\mathbb E}}
\renewcommand{\P} {{\mathbb P}}

\newcommand{\R} {{\mathbb R}}

\newcommand{\dt}{\tau}
\newcommand{\dx}{h}
\newcommand{\conv}{\rightarrow}
\newcommand{\atan}{\mbox{atan}}
\newcommand{\disp}{\displaystyle}

\newcommand{\be}{\begin{eqnarray}}
\newcommand{\ee}{\end{eqnarray}}

\newcommand{\beno}{\begin{eqnarray*}}
\newcommand{\eeno}{\end{eqnarray*}}

\newcommand{\<}{\langle}
\renewcommand{\>}{\rangle}

\newcommand{\barr}[2]{\begin{array}{#1} #2}
\newcommand{\earr}{\end{array}}

\DeclareMathOperator{\tridiag}{tridiag}
\DeclareMathOperator{\diag}{diag}
\newcommand{\xmin}{X_{min}}
\newcommand{\xmax}{X_{max}}
\newcommand{\oldbcoeff}{{\delta}}
\newcommand{\linconnu}{{y}}
\newcommand{\dop}{\mathrm{d}}
\newcommand{\ds}{\dop s}
\newcommand{\du}{\dop u}
\DeclareMathOperator{\pdiag}{pdiag}

\newcommand{\xhat}{\hat{x}}

\begin{document}

\title
[{BDF finite difference schemes for diffusion equations with obstacle}]
{{Backward Differentiation Formula finite difference schemes for diffusion equations with an obstacle term}}

\author
[O.\ Bokanowski, K.\ Debrabant]
{
Olivier Bokanowski\address{
Laboratoire Jacques-Louis Lions,
Universit{\'e} de Paris (Paris Diderot)
Paris, France
and
Ensta ParisTech
(\textit{olivier.bokanowski@math.univ-paris-diderot.fr})}
\and
Kristian Debrabant\address{
University of Southern Denmark,
Department of Mathematics and Computer Science, Odense
(\textit{debrabant@imada.sdu.dk})}
}

\begin{abstract}
Finite difference schemes, using Backward Differentiation Formula (BDF),
are studied for the approximation of one-dimensional diffusion equations with an obstacle term, of the form
$$\min(v_t - a(t,x) v_{xx} + b(t,x) v_x + r(t,x) v, v- \varphi(t,x))= f(t,x).$$
{For the scheme building on {the second order BDF formula} {(BDF2)},
we discuss unconditional stability, prove an $L^2$-error estimate
and show numerically second order convergence, in both space and time,
unconditionally on the ratio of the mesh steps.
In the analysis, an equivalence of the obstacle equation with a Hamilton-Jacobi-Bellman equation
is mentioned, and a Crank-Nicolson scheme is tested in this context.
Two academic problems for parabolic equations with an obstacle term
with explicit solutions and the American option problem in mathematical
finance are used for numerical tests.}
\end {abstract}

\maketitle


\section{Introduction}
We consider a second order partial differential equation with an obstacle term, of the following form:
\begin{subequations}\label{eq:1}
\be 
  & & \min (v_t + \cA v, v -\varphi(t,x))=f(t,x), \quad t\in(0,T), \quad x \in \mO, \label{eq:1a}\\
  & & v(0,x)=v_0(x), \quad x\in \mO, \label{eq:1b}
\ee
\end{subequations}
with
\be \label{eq:Aopp-gen}
  \cA v:= -\frac{1}{2} \ms^2(t,x) v_{xx} + b(t,x) v_x + r(t,x) v.
\ee
{We will assume that $b$, $r$, $\ms$, $f$, $\varphi$ and $v_0$ are Lipschitz continuous functions with respect to all variables,
and also $v_0(x)\geq \varphi(0,x)+f(0,x)$ for compatibility reasons with \eqref{eq:1a},
and $\mO$ is a subset of $\R$.}

{When $\mO=\R$, the solution $v$ can be defined as the unique uniformly continuous viscosity solution of \eqref{eq:1}
in the viscosity sense (see \cite{Pham_1998} for a precise statement, for the case of $x$-dependent obstacle functions and $f\equiv0$, using even less restrictive
assumptions on the remaining data,
and these results can easily be generalized to the case of $(t,x)$-dependent obstacle functions and $f\not\equiv0)$.
The PDE \eqref{eq:1} can also be considered on a bounded domain $\mO=(\xmin,\xmax)$ with Dirichlet boundary conditions,
see \cite{Lions_1983a,Lions_1983b,CraIshLio92} and Section~\ref{sec:CN-FD-schemes}.
For the well-posedness of \eqref{eq:1}, a variational framework can also be used~\cite{friedman-88,ach-pir-2005}.
}

In the recent years there has been a lot of interest in the approximation of such obstacle problems.
Related to {stochastic optimal} stopping time problems, we will consider in particular
\begin{subequations} \label{eq:amer-pb}
\begin{gather}
  \min(v_t- \frac12\ml^2 x^2 v_{xx} - r x v_x + r v,\ v-\varphi(x))=0,\quad t\in(0,T),\ x\in\mO,\\
  v(0,x)=\varphi(x),\quad x\in \mO,
\end{gather}
\end{subequations}
with $\mO=(0,\infty)$,
with constant coefficients $\lambda>0$, $r>0$, $f=0$,
and with initial data identical to the obstacle function.
The American put option problem in mathematical finance corresponds in particular to the case of the initial data (or \enquote{payoff} function)
$ \varphi(x):=\max(K-x,0)$.
For this problem it is known
that the solution presents a singular point $x_s(t)$
moving with time, such that $v(t,x)=\varphi(x)$ for $x\leq x_s(t)$, and $v(t,x)>\varphi(x)$ for $x>x_s(t)$,
and $t\conv x_s(t)$ has a H\"{o}lder continuity behavior near $t=0$
(see Remark~\ref{rem:xs-amer} as well as
\cite{dew-how-rup-wil-1993},  \cite{bar-bur-rom-sam-1993}, and \cite[Chap 6]{ach-pir-2005}).
Some results on the structure of the interface related to \eqref{eq:1} can also be found in \cite{bla-dol-mon-06}
(and see related references).

A finite element scheme in the American option setting has been considered by
Jaillet, Lamberton and Lapeyre in \cite{jai-lam-lap-90}, where its convergence is also proved under conditions on the mesh steps.
For a comprehensive study of finite difference schemes as well as finite element schemes in this context we refer
to Achdou and Pironneau~\cite{ach-pir-2005}.
In relation with the obstacle problem, a finite volume method is also studied in Berton and Eymard \cite{ber-eym-2006}.
In connection with the present work,
{Windcliff, Forsyth, and Vetzal applied in \cite{windcliff01soa} a second order backward differentiation formula (BDF2) scheme to shout options, which can be understood as a sequence of interdependent American option type problems. Also} Oosterlee \cite{oos-2003} applied BDF2 in the context of the American option problem, in combination with a multigrid approach (see also {Oosterlee} et al.\ \cite{oos-gas-fri-2003}).
{Le Floc'h \cite{lefloch14tbf} applied the trapezoidal rule combined with BDF2 as a one-step method (TR-BDF2) to the American option problem.}

{In relation with viscosity theory,
a precise error analysis is given in \cite{jak-2003} for {monotone} finite difference schemes and semi-Lagrangian schemes.}

Now we remark that in the case of $v_0\equiv \varphi{+f}$ (with $\mO=\R$) {in \eqref{eq:1b}}, and for an operator of the form
\be \label{eq:Aopp}
  \cA v:= -\frac{1}{2} \ms^2(x) v_{xx} + b(x) v_x + r(x) v,
\ee
i.\,e., with coefficients and source term $f=f(x)$ which do not depend on time {(and which are otherwise Lipschitz continuous)}
the solution $v$ of \eqref{eq:1} is also the unique viscosity solution of the following Hamilton-Jacobi-{Bellman} (HJ{B})
equation:
\begin{subequations}\label{eq:2}
\be 
  & & v_t + \min(0,\cA v)=f(x), \quad t\in(0,T), \quad x \in \mO,\\
  & & v(0,x)=\varphi(x)+f(x), \quad x\in \mO.
\ee
\end{subequations}
{(For the well-posedness of \eqref{eq:2} in the viscosity framework, see~\cite{Pham_1998}.)}
The equivalence between \eqref{eq:1} and \eqref{eq:2} was signaled to us by R.~Eymard.
It is proved {in  Martini}~\cite{martini:inria-00072718}.
A sketch of the proof of independent interest is given
in Appendix~\ref{app:proof-of-HJB} (see Remark~\ref{rem:equivalence-with-nonzero-f}).

{
In this article,
we first study in Section \ref{sec:CN-FD-schemes}
two elementary Crank-Nicolson (CN) schemes: a classical CN scheme adapted to the obstacle problem \eqref{eq:1}, and an other CN scheme
adapted to the PDE~\eqref{eq:2}.
Although these CN schemes are both second order consistent (and their results even agree in certain cases),
we numerically observe that they tend to switch back to first order convergence for bad ratios of the mesh parameters
(corresponding to large time steps or \enquote{high CFL numbers}).
}
{It is known that a change of variable in time, as in \cite{Reisinger_Whitley_2014},
or the use of refined time steps near singularity (i.\,e.\ near $t=0$), as in \cite{Forsyth_Vetzal_2002},
can be used as a remedy to recover second order convergence. However these remedies somehow correspond to using smaller time steps,
which one may want to avoid.}
Stability results exist for the {CN} scheme, in the $L^\infty$ setting,
for the approximation of the linear heat equation \cite{serdjukova67uso}.
However, to the best of our knowledge, there is no convergence proof for the {CN} finite difference schemes
adapted to the {obstacle} equation \eqref{eq:1}
{without assuming a CFL condition of the form $\dt/h^2$ small enough (where $\tau$ is the time step and $h$ a mesh step).}

To circumvent some of these problems,
we then consider the use of
the Backward {Differentiation} Formula (BDF) for the approximation of the time derivative $v_t$, adapted to the obstacle problem.
In Section {\ref{sec:BDFimp}}, we introduce implicit BDF schemes in the same way as {Windcliff, Forsyth, and Vetzal \cite{windcliff01soa} and} Oosterlee \cite{oos-2003}.
In particular second- and third-order BDF schemes are considered. 
When formulated on {an obstacle problem} \eqref{eq:1}, these schemes are non-linear and implicit,
but they can be solved by using a simple Newton-type algorithm.
{In Section \ref{sec:convEstimates}}, a new unconditional $L^2$ stability estimate is obtained in the case of the second order BDF obstacle scheme (BDF2).
This is {achieved} by using estimates similar to the
\enquote{Gear} scheme for {parabolic PDEs}
(see for instance \cite{ach-pir-2005}, {see also \cite{crouzeix84and}})
with some new ingredients in order to deal with the non-linearity coming from the obstacle term.
We then obtain also a new error estimate in an $L^2$ norm.
This estimate holds under some specific assumptions on the regularity of
the exact solution $v$ which allows $v_{xx}$ bounded but possibly discontinuous at some finite number of singular points
$(t,y_j(t))_{1\leq j\leq p}$ that do not evolve too rapidly.

{In Section \ref{sec:numericalresults}}, two academic models are introduced, with explicit solutions,
{one of them being} very close to the American option model.
These models allow us to study precisely and more easily the  numerical convergence
and allow for a slightly smoother behavior of the interface (compared to the American option model).
This allows also to observe third order behavior for a {third order BDF obstacle scheme} {(}BDF3{)}
on a specific model with bounded $u_{xxx}$ derivative at the free boundary. The MATLAB source code for all numerical experiments in this manuscript can be found at \cite{BokaDebra2020}.

{Appendix \ref{app:proof-of-HJB} is devoted to a sketch of the proof for
the equivalence between PDE \eqref{eq:1} and PDE \eqref{eq:2} in case
the coefficients are not time dependent.}

Our study concerns here only one-dimensional obstacle problems, but
the proposed schemes based on BDF approximations can be extended to higher dimensions
(see \cite{oos-2003,Bokanowski_Picarelli_Reisinger_2018}).

\medskip

{
\paragraph{\bf Acknowledgements.} We are very grateful for the many helpful comments of the referees,
especially concerning the analysis and the useful references related to the CN scheme, and for remarks that
helped us improve the convergence result for the BDF2 scheme.
}%

\section{{CN} finite difference schemes revisited}\label{sec:CN-FD-schemes}

In this section we revisit the {CN} schemes and related approaches for a diffusion equation in presence of an
obstacle term.
{
Although the presented schemes are all theoretically second order consistent in smooth regions (in a sense that is made precise
in Lemma~\ref{lem:2.3}),
we will show that
the order may numerically deteriorate and switch back
to first order for \enquote{high CFL} numbers, i.\,e., for large time steps with respect to space steps.
(As mentioned in the introduction, a change of variable in time, as in \cite{Reisinger_Whitley_2014},
or the use of refined time steps near the $t=0$ singularity, as in \cite{Forsyth_Vetzal_2002}, can lead back to second order convergence).
The BDF scheme presented in Section \ref{sec:BDFimp} will not suffer from this problem.
}%

For the numerical approximation of \eqref{eq:1} we will consider $\mO=(\xmin,\xmax)$ together with Dirichlet boundary conditions:
\begin{subequations}\label{eq:dirichlet}
\be
   & & v(t,\xmin)=v_{\ell}(t), \quad  t\in(0,T),\\
   & & v(t,\xmax)=v_{r}(t),\quad t\in (0,T).
\ee
\end{subequations}
We consider a uniform mesh with $J\geq 1$ points inside:
$$
  x_j=X_{min} + j h,\quad j=0,\dots,J+1,
$$
where $h:=\frac{\xmax-\xmin}{J+1}$.
Let $N\geq 1$, $\dt=\frac{T}{N}$ and $t_n = n\dt$.

We shall say that we have a \enquote{high CFL number} when $\frac{\tau}{h}\gg 1$ (or $\frac{J}{N}\gg 1$),
compared to a situation where $\frac{\tau}{h}\simeq 1$ (or $\frac{J}{N}\simeq 1$).

Denoting ${v^n_j}{:}=v(t_n,x_j)$,
we consider {the following} centered finite difference approximation for the operator $\cA v$:
\be\label{eq:A}
  (\cA v)(t_n,x_j) & \simeq  & \\
   & & \hspace{-2cm}
    \frac{1}{2} \sigma^2(t_n,x_j)
    \bigg(\frac{-{v}^n_{j-1}+2 {v}^n_j - {v}^n_{j+1}}{h^2}\bigg) + b(t_n,x_j)\frac{{v}^{n}_{j+1}-{v}^n_{j-1}}{2h} + r(t_n,x_j) {v}^n_j.
    \nonumber
\ee

The diffusion part will always dominate the advection part ($\frac{1}{2}\frac{\ms^2}{h^2}\geq \frac{|b|}{2h}$)
to avoid stability issues with the centered approximation.
Note that for the American put option problem this requires $h\leq x_{1} {\lambda}^2/r$.

{
\begin{rem}
4th order finite difference approximations for $v_x$ and $v_{xx}$ can also be used instead of \eqref{eq:A}, in particular for the numerical tests, as detailed in Section~\ref{sec:5.2}.
\end{rem}
}

Let us denote by {$A^{(n)} u^n + q^n$} the approximation of $\cA v(t_n,\cdot)$
on a given set of grid points, with $u^n=(u^n_1,\dots, u^n_J)^T$, {where $u^n_j$ are approximations of ${v^n_j}$},
{$A^{(n)}=(a^{(n)}_{i,j})_{1\leq i,j\leq J}$ with
\begin{subequations}\label{eq:A2ndorder}
\be
  a^{(n)}_{i,i-1} & := & - \mb^n_i -\mg^n_i,       \quad i=2,\dots,J, \\
  a^{(n)}_{i,i}   & := &   2\mb^n_i  + r(t_n,x_i), \quad i=1,\dots,J,  \\
  a^{(n)}_{i,i+1} & := & - \mb^n_i +\mg^n_i,       \quad i=1,\dots,J-1,
\ee
\end{subequations}
where $\mb^n_i:=\frac1{2h^2}\sigma^2(t_n,x_i)$ and $\mg^n_i:=\frac{b(t_n,x_i)}{2h}$,
and
\beno
  q^n:=\big((-\mb^n_1-\mg^n_1) u_0^n,
      0, \dots, 0, (-\mb^n_J+\mg^n_J) u_{J+1}^n\big)^\top,
\eeno
}
and with given Dirichlet boundary conditions,
for $n=0,\dots,N$:
\be \label{eq:bc}
  u^n_0=v_\ell(t_n)\quad \mbox{and}\quad  u^n_{J+1}=v_r(t_n).
\ee
{
The matrix $A^{(n)}$ is in general time-dependent, but
for simplicity of presentation we will write $A^{(n)}\equiv A \equiv (a_{i,j})$ without explicit time-dependency.}
The vector $q^n$ may depend on the time also because of the time dependency in the boundary conditions~\eqref{eq:bc}.
We have second order consistency in space, that is, {assuming} $v$ sufficiently regular,
$$
  (A v^n + q^n)_j =  (\cA v)(t_n,x_j) + O(\dx^2).
$$

A first simple {CN} scheme for the obstacle equation \eqref{eq:1}
is, for $n=0,\dots,N-1$ and  $j=1,\dots,J$, given by
\begin{multline}\label{eq:CNscheme-1}
  \cS^{1,n}_j{(u)}
  :=
  \min\bigg(
    \frac{u^{n+1}_j - u^n_j}{\dt}  + \frac{1}{2} ( A u^{n+1} + A u^n)_j + q^{n+1/2}_j {-f^{n+1/2}_j} ,\\
     u^{n+1}_j {-{\varphi^{n+1}_j} -f^{n+1}_j}
    \bigg) = 0,
\end{multline}
where we have denoted
$$
  {
  {\varphi^{n+1}_j}:=\varphi(t_{n+1},x_j)\quad \mbox{and}\quad
  f^{p}_j:=f(t_p,x_j),\ p\in\{n+\frac 12,n+1\}
  }
$$
{and use} the boundary conditions \eqref{eq:bc} and {initial condition}
\be\label{eq:ic}
  u^0_j:=v_0(x_j), \quad 1\leq j\leq J.
\ee

Looking now at equation \eqref{eq:2},
an other possible CN scheme is
\begin{multline}\label{eq:CNscheme-2}
  \cS^{2,n}_j{(u)}  :=
    \min\bigg(
    \frac{u^{n+1}_j - u^n_j}{\dt}  + \frac{1}{2} ( A u^{n+1} + A u^n)_j + q^{n+1/2}_j {- f^{n+1/2}_j},\\
     \frac{u^{n+1}_j - u^n_j}{\dt} {- f^{n+1/2}_j}
    \bigg)  = 0,\qquad 1\leq j\leq J,
\end{multline}
initialized with $u^0_j:=v_0(x_j)$.
Because $\dt>0$, this scheme is also equivalent to
\begin{multline} \label{eq:CNscheme-2bis}
  \min\bigg(
    \frac{u^{n+1}_j - u^n_j}{\dt}  + \frac{1}{2} ( A u^{n+1} + A u^n)_j + q^{n+1/2}_j {- f^{n+1/2}_j},\\
     u^{n+1}_j - u^n_j {- \dt f^{n+1/2}_j}
    \bigg) = 0,
  \qquad1\leq j\leq J.
\end{multline}

\begin{rem}
For both schemes, the unknown $u^{n+1}$ is {unique, well defined,} and can be obtained by using fix point methods.
{One can use the fact that} there exists a unique solution
of the obstacle problem $\min(Bx-\oldbcoeff,x-g)=0$ as soon as, for instance, $B$ is {strictly} diagonally dominant with $B_{ii}>0$
(see \cite{ach-pir-2005}).
\end{rem}

\begin{rem}\label{eq:newton}
If furthermore $B$ is a {strictly diagonally dominant $M$-matrix}
(i.\,e.\ $B_{ij}\leq 0$ for all $i\neq j$ and $B_{ii}>\sum_{j\neq i} |B_{ij}|$ for all $i$)
then a Newton-like algorithm \cite{bok-mar-zid-2009} can be implemented, which is particularly
efficient for solving obstacle problems exactly (up to machine precision) in a few number of iterations. We refer also to \cite{hin-ito-kun-2003}
for convergence of semi-smooth Newton methods or related algorithms applied to solve discretized PDE obstacle problems.
{Note that the number of iterations before convergence may increase linearly with the number of mesh points \cite{bok-mar-zid-2009}.
Penalization methods for solving the obstacle problem
can also be used in order to approximate the equation with a controlled penalization error and then
significantly reduce the number of needed iterations~\cite{Forsyth_Vetzal_2002,Witte_Reisinger_2012,Reisinger_Zhang_Preprint2018}.}
\end{rem}

{
For the linear part $v_t+\cA v$, the CN scheme appears to be second order consistent only precisely at time $t_{n+1/2}$,
while the obstacle term $v^{n+1}_j-{\varphi^{n+1}_j}$ is evaluated at time $t_{n+1}$ in \eqref{eq:CNscheme-1}, so appears to be only first order consistent with
the value $v(t_{n+1/2},x_j)-\varphi(t_{n+1/2},x_j)$. Hence the consistency error seems to be in general of order $O(\tau) + O(\dx^2)$ but not better.
The following lemma shows that a better order may hold.}
\begin{lemma}\label{lem:2.3}
{
$(i)$
The CN scheme \eqref{eq:CNscheme-1} is second order consistent in time and space, in the following sense:
for any regular $v$ that is solution to \eqref{eq:1a} it holds
$$ {\cS}^{1,n}_j(v)= O(\dt^2 + h^2).$$
$(ii)$ The CN scheme \eqref{eq:CNscheme-2} is second order consistent in time and space, in the following sense:
for any regular $v$ it holds}
\[{ {\cS}^{2,n}_j(v)= \min(v_t + \cA v-f, v_t-f)(t_n,x_j)  + O(\dt^2+ h^2).}\]
\end{lemma}
\begin{proof}
$(i)$
{$v$ regular implies}
\begin{multline}\label{eq:consistencylinearpart}
 \frac{v^{n+1}_j - v^n_j}{\dt}   + \frac{1}{2} ( A v^{n+1} + A v^n)_j + q^{n+1/2}_j\\
 = (v_t + \cA v)(t_{n+1/2},x_j)   + O(\dt^2 \| v_{3t}\|_\infty) + O(\dx^2 (\| v_{3x}\|_\infty + \| v_{4x}\|_\infty)).
\end{multline}

{
Inserting \eqref{eq:consistencylinearpart} into \eqref{eq:CNscheme-1} we obtain the estimate
\begin{multline}
  {\cS}^{1,n}_j(v)=\min( (v_t + \cA v)(t_{n+1/2},x_j) - f^{n+1/2}_j,\\ v(t_{n+1},x_j)-{\varphi^{n+1}_j}-f^{n+1}_j) + O(\dt^2+ \dx^2).
  \label{eq:1n12}
\end{multline}

Now we consider three possible cases.
First case: $v(t_{n+1},x_j)={\varphi^{n+1}_j}+f^{n+1}_j$.
Since $v_t + \cA v - f \geq 0$ by \eqref{eq:1}, 
it follows ${\cS}^{1,n}_j(v)= O(\dt^2+ \dx^2)$.

Second case: $v(t_{n+1},x_j)>{\varphi^{n+1}_j}+f^{n+1}_j$ and $v(t_{n+1/2},x_j)>\varphi^{n+1/2}_j+f^{n+1/2}_j$.
In that case, since $v$ is solution to \eqref{eq:1} at $(t_{n+1/2},x_j)$,
it holds \[(v_t + \cA v - f)(t_{n+1/2},x_j)=0,\] using \eqref{eq:1n12} yields therefore
\beno
  {\cS}^{1,n}_j(v)
    & = & \min(0 ,\ v(t_{n+1},x_j)-{\varphi^{n+1}_j}-f^{n+1}_j) + O(\dt^2+ \dx^2)   \\
    & = & O(\dt^2+ \dx^2).
\eeno

Third case: $v(t_{n+1},x_j)>{\varphi^{n+1}_j}+f^{n+1}_j$ and $v(t_{n+1/2},x_j)=\varphi^{n+1/2}_j+f^{n+1/2}_j$.
Since $v(t,x_j)\geq \varphi(t,x_j) + f(t,x_j)$ for all $t$, $t\conv v(t,x_j)-\varphi(t,x_j)-f(t,x_j)$ reaches a minimum at $t=t_{n+1/2}$,
and by using the regularity of $v$, $f$ and $\varphi$  we obtain $v_t(t_{n+1/2},x_j)-\varphi_t(t_{n+1/2},x_j)-f_t(t_{n+1/2},x_j)=0$.
Then $v(t_{n+1},x_j)-\varphi(t_{n+1},x_j)-f(t_{n+1},x_j)=v(t_{n+1/2},x_j) -\varphi(t_{n+1/2},x_j)-f(t_{n+1/2},x_j) + O(\tau^2)= O(\tau^2)$,
from which we deduce
$$
  {\cS}^{1,n}_j(v)=\min( (v_t + \cA v - f)(t_{n+1/2},x_j),\ 0) + O(\dt^2+ \dx^2) \ = O(\dt^2+ \dx^2).
$$
}
$(ii)$
Inserting \eqref{eq:consistencylinearpart} and
\[\frac{v^{n+1}_j- v^n_j}{\dt} = v_t (t_{n+1/2},x_j) + O(\dt^2 \| v_{3t}\|_\infty)\]
into \eqref{eq:CNscheme-2} we obtain the second order estimate
\[
  {\cS}^{2,n}_j(v)=\min(v_t + \cA v - f, v_t - f)(t_{n+1/2},x_j) + O(\dt^2+ \dx^2).
\]%
\end{proof}
{%
\begin{rem}
In the proof of $(i)$ we use the fact that $v$ is solution of the PDE, as well as
the regularity of $t\conv v(t,x_j)$ in a region where it switches from $v(t_n,x_j)=\varphi^n_j$ to $v(t_n,x_j)>\varphi^n_j$,
and we know in general that this corresponds to a jump of $v_t$ (or $v_{xx}$). Therefore this
analysis cannot be applied in general.
If $v$ is regular, without assuming that $v$ is solution of \eqref{eq:1},
then by \eqref{eq:1n12} we would obtain only a first order estimate in time.
\end{rem}
}

{On the other hand,} the following assertions hold:

\begin{lemma}\label{lem:2.2}
{Assume that $f$ and $\varphi$ are independent of $t$.}

$(i)$ {If for given $u^n$ value,
the solution $u^{n+1}$ of the scheme \eqref{eq:CNscheme-1} satisfies $u^{n+1}\geq u^n$ (with $u^n\geq \varphi+f$), then
$u^{n+1}$ is also solution of the scheme \eqref{eq:CNscheme-2} starting from $u^{n}$.}

$(ii)$ {In particular, if {$u^1\geq u^0$} and the following conditions hold:}
\begin{itemize}
\item
  the vector $q$ {and the matrix $A$ do} not depend on time,
\item
  the matrix $I-\frac{\dt}{2} A$ is positive componentwise,
\item
  the matrix $I+\frac{\dt}{2}A$ is a {strictly diagonally dominant} $M$-matrix,
\end{itemize}
{then the solution of the scheme \eqref{eq:CNscheme-1} satisfies $u^{n+1}\geq u^n$ for all $n$, and thus by (i) schemes \eqref{eq:CNscheme-1} and \eqref{eq:CNscheme-2} give identical values.}
\end{lemma}

\begin{proof}
Proof of $(i)$:
Let $c^{{n}}:=\frac{1}{2} ( A u^{n+1} + A u^n)_j + q^{n+1/2}_j{-f_j}$, so that the first scheme \eqref{eq:CNscheme-1}
reads $\min(\frac{u^{n+1}_j - u^n_j}{\dt}  + c^n_j, \ u^{n+1}_j{-\varphi_j-f_j})=0$.
Assuming that $u^{n+1}$ is solution of scheme \eqref{eq:CNscheme-1}, with
$u^{n+1}\geq u^{n}$, if $u^{n+1}_j={\varphi_j+f_j}$, then it is clear that also $u^n_j={\varphi_j+f_j}=u^{n+1}_j$ and therefore since
 $\frac{u^{n+1}_j - u^n_j}{\dt}  + c^n_j\geq 0$
it can be deduced that ${\cS}^{2,n}_j(u)=0$. On the other hand, if $u^{n+1}_j>{\varphi_j+f_j}$, then
it implies $\frac{u^{n+1}_j - u^n_j}{\dt}  + c^n_j=0$, from which we conclude $\cS^{2,n}_j(u)=0$, so $u^{n+1}$
is also a solution of the scheme \eqref{eq:CNscheme-2}.

Proof of $(ii)$: {
Denoting $F_n(x):=\min(B x - \oldbcoeff_n, x-{\varphi-f})$ where $B=I+\frac{\dt}{2} A$
and $\oldbcoeff_n = (I-\frac{\dt}{2}A) u^n - \dt q {-\dt f}$, the scheme \eqref{eq:CNscheme-1} is equivalent to $F_n(u^{n+1})=0$. Because $B$ is an $M$-matrix,
the function $F_n$ is monotone in the sense that $F_n(y)\leq F_n(z)$ $\Rightarrow$ $y \leq z$ (componentwise).

Assume now that $u^{n}\geq u^{n-1}$ for some $n\geq 1$. Due to $I-\frac{\dt}{2}A\geq 0$, it holds $\oldbcoeff_n\geq \oldbcoeff_{n-1}$, and therefore
$F_n(y)\leq F_{n-1}(y)$ for all $y$. In particular, $F_{n-1}(u^{n})=0=F_n(u^{n+1}) \leq F_{n-1}(u^{n+1})$, and by the monotonicity
of $F_{n-1}$ we conclude $u_{n}\leq u_{n+1}$. By induction it follows that $u^{n+1}\geq u^n$ for all $n$.}
\end{proof}
{
\begin{rem}\label{rem:Lem2.2forAmericanPutoptionproblem}
For the American put option problem \eqref{eq:amer-pb} with $\varphi(x):=\max(K-x,0)$ the left boundary condition will be $u^n_0=K-X_{min}$,
and the right boundary condition $ u^n_{J+1}=0$. Thus the vector $q^{n+1/2}=q$ does not depend on time.
Further, for $A$ determined by \eqref{eq:A} {and $\lambda^2(X_{\min}+\dx)>r\dx$}, the matrix $I-\frac{\dt}{2} A$ is positive componentwise under the CFL condition $\left(\frac{{\lambda^2X_{\max}^2}}{h^2}+r\right)\frac{\dt}{2}\leq 1$,
and the matrix $I+\frac{\dt}{2}A$ is a strictly diagonally dominant $M$-matrix {under the condition $\frac{\dt}{\dx}<\frac{2+\dt r}{rX_{\max}}$}.
Finally, with $u^0=g$, it is easy to see that the scheme \eqref{eq:CNscheme-1} satisfies $u^1\geq u^0=g$.
Thus, by Lemma \ref{lem:2.2}, scheme \eqref{eq:CNscheme-1} and scheme \eqref{eq:CNscheme-2} give identical values.
\end{rem}
}

This explains why {the CN scheme \eqref{eq:CNscheme-1} gives the same results as scheme \eqref{eq:CNscheme-2} for low CFL number}.

In order to verify the expected orders,
we have tested the CN schemes numerically on the American put option model
with initial data
\begin{subequations}
\label{eq:parameters}
\be
  \varphi(x):=\max(K-x,0)
\ee
and parameters
\be
  \lambda=0.2, \quad r=0.1, \quad T=1, \quad K=100.
\ee
\end{subequations}

In this setting, we observe that the singular point $x_s(T)$ is greater than $80$, so
for the numerical approximation we have considered the subdomain $\mO=(\xmin,\xmax)\equiv(75,275)$
and boundary conditions of Dirichlet type:
$$
  v(t,X_{\min})=K-X_{\min}, \quad 0<t<T,
$$
and
$$
  v(t,X_{\max})=0, \quad 0<t<T.
$$
We have numerically estimated that the truncation error from the right, using $X_{max}=275$ instead of $X_{max}=+\infty$,
is less than $10^{-8}$.

{The errors {of the CN schemes} in $L^2$, $L^1$ and $L^\infty$ norms are computed at time $t_N=T$ as follows:
\be\label{discreteLperrordef}
  e_{L^p}:=(h\sum_{{j=1}}^{{J}} |u^N_{{j}}-v^N_{{j}}|^p)^{1/p}
  \mbox{ and }e_{L^\infty}{:}=\max_{{{j=1}}}^{{{J}}} |u^N_{{j}}-v^N_{{j}}|.
\ee
}%
{The} reference values are computed using a BDF obstacle scheme of {second} order and with {$N=J{+1}=20480$}
that will be made precise in Section~\ref{sec:BDFimp}. {In the Newton-like algorithm to solve the obstacle problem $\min(Bx-\oldbcoeff,x-g)=0$, we iterated until the numerical approximation $\xhat$ fulfilled $\|\min(B\xhat-\oldbcoeff,\xhat-g)\|_{\infty}<10^{-10}$.}

Results are given in Table~\ref{tab:CN_highCFL} with discretization parameters $N=J{+1}$ and $N={(}J{+1)}/10$
(this second case corresponds to large time steps or \enquote{high CFL numbers}).
Note that the quite restrictive CFL condition in part (ii) of Lemma \ref{lem:2.2} {(cmp.\ Remark \ref{rem:Lem2.2forAmericanPutoptionproblem})}
is not fulfilled.

However, for lower $N$ values (higher CFL numbers) the CN scheme is no more second order and goes back to first order
behavior. This is illustrated in Table~\ref{tab:CN_highCFL}. In particular we observe that the pointwise inequality
$u^{n+1}\geq u^n$ is no more true (due to the fact that
the amplification matrix $(I+\frac{\tau}{2} A)^{-1} (I-\frac{\tau}{2} A)$
does not have only positive coefficients anymore).

Now, taking the obstacle to be $u^n$ instead of {$\varphi$}, hence solving the scheme \eqref{eq:CNscheme-2},
enforces that $u^{n+1}\geq u^n$. Results for the scheme \eqref{eq:CNscheme-2} are similar to that for scheme \eqref{eq:CNscheme-1}
for low CFL numbers (both schemes give identical values for the case $N=J{+1}$, here).
For higher CFL numbers, results obtained with the scheme \eqref{eq:CNscheme-2} {differ (see Table~\ref{tab:CN-2-highCFL}),
but again switch back to first order behavior. They are even less precise than the CN scheme~\eqref{eq:CNscheme-1}
for the $L^1$ and the $L^2$ errors.}

\begin{table}[!hbtp]
\begin{center}
\begin{tabular}{|cc|cc|cc|cc|c|}
\hline
\multicolumn{2}{|c|}{Mesh}
   & \multicolumn{2}{|c|}{Error $L^1$} & \multicolumn{2}{|c|}{Error $L^2$} & \multicolumn{2}{|c|}{Error $L^\infty$}
   &  time(s)
   \\
\hline
  $J{+1}$ & $N$   &  error   &  order &    error & order  & error    & order   &         \\
\hline\hline
    80 &     80 &{7.21e-01} &{ 1.88 }& {1.27e-01}& {1.80 }& {4.49e-02 }& {1.20 }& { 0.01}\\
   160 &    160 &{1.42e-01} &{ 2.35 }& {2.28e-02}& {2.48 }& {5.29e-03 }& {3.09 }& { 0.01}\\
   320 &    320 &{3.79e-02} &{ 1.90 }& {6.04e-03}& {1.92 }& {1.40e-03 }& {1.92 }& { 0.04}\\
   640 &    640 &{1.01e-02} &{ 1.91 }& {1.58e-03}& {1.93 }& {3.57e-04 }& {1.97 }& { 0.12}\\
  1280 &   1280 &{2.79e-03} &{ 1.85 }& {4.29e-04}& {1.88 }& {9.21e-05 }& {1.95 }& { 0.45}\\
  2560 &   2560 &{7.98e-04} &{ 1.80 }& {1.20e-04}& {1.84 }& {2.76e-05 }& {1.74 }& { 1.72}\\
  5120 &   5120 &{2.20e-04} &{ 1.86 }& {3.24e-05}& {1.89 }& {6.48e-06 }& {2.09 }& { 7.20}\\
\hline\hline
    80 &      8 &{7.93e-01} &{ 1.66 }& {1.45e-01}& {1.62 }& {4.90e-02 }& {1.51 }& { 0.00}\\
   160 &     16 &{2.00e-01} &{ 1.99 }& {3.62e-02}& {2.00 }& {2.18e-02 }& {1.17 }& { 0.00}\\
   320 &     32 &{6.40e-02} &{ 1.64 }& {1.21e-02}& {1.58 }& {9.91e-03 }& {1.14 }& { 0.01}\\
   640 &     64 &{2.31e-02} &{ 1.47 }& {4.38e-03}& {1.47 }& {4.75e-03 }& {1.06 }& { 0.03}\\
  1280 &    128 &{8.71e-03} &{ 1.41 }& {1.61e-03}& {1.44 }& {2.32e-03 }& {1.04 }& { 0.12}\\
  2560 &    256 &{3.55e-03} &{ 1.29 }& {6.22e-04}& {1.37 }& {1.14e-03 }& {1.02 }& { 0.62}\\
  5120 &    512 &{1.50e-03} &{ 1.24 }& {2.49e-04}& {1.32 }& {5.65e-04 }& {1.01 }& { 3.41}\\
\hline
\end{tabular}
\end{center}
\caption{\label{tab:CN_highCFL} {CN} scheme \eqref{eq:CNscheme-1}
with different mesh parameters $N=J{+1}$ and $N={(}J{+1)}/10$.}
\end{table}

\begin{table}[!hbtp]
\begin{center}
\begin{tabular}{|cc|cc|cc|cc|c|}
\hline
\multicolumn{2}{|c|}{Mesh}
   & \multicolumn{2}{|c|}{Error $L^1$} & \multicolumn{2}{|c|}{Error $L^2$} & \multicolumn{2}{|c|}{Error $L^\infty$}
   &  time(s)
   \\
\hline
  $J{+1}$ & $N$   &  error   &  order &    error & order  & error    & order   &         \\
\hline\hline
    80 &     80 &{7.21e-01} &{ 1.88 }& {1.27e-01}& {1.80 }& {4.49e-02 }& {1.20 }& { 0.01}\\
   160 &    160 &{1.42e-01} &{ 2.35 }& {2.28e-02}& {2.48 }& {5.29e-03 }& {3.09 }& { 0.02}\\
   320 &    320 &{3.79e-02} &{ 1.90 }& {6.04e-03}& {1.92 }& {1.40e-03 }& {1.92 }& { 0.04}\\
   640 &    640 &{1.01e-02} &{ 1.91 }& {1.58e-03}& {1.93 }& {3.57e-04 }& {1.97 }& { 0.13}\\
  1280 &   1280 &{2.79e-03} &{ 1.85 }& {4.29e-04}& {1.88 }& {9.21e-05 }& {1.95 }& { 0.44}\\
  2560 &   2560 &{7.98e-04} &{ 1.80 }& {1.20e-04}& {1.84 }& {2.76e-05 }& {1.74 }& { 1.79}\\
  5120 &   5120 &{2.20e-04} &{ 1.86 }& {3.24e-05}& {1.89 }& {6.48e-06 }& {2.09 }& { 8.30}\\
\hline\hline
    80 &      8 &{7.04e-01} &{ 1.83 }& {1.35e-01}& {1.73 }& {4.90e-02 }& {1.51 }& { 0.00}\\
   160 &     16 &{8.05e-02} &{ 3.13 }& {2.52e-02}& {2.42 }& {1.61e-02 }& {1.60 }& { 0.00}\\
   320 &     32 &{5.17e-02} &{ 0.64 }& {1.05e-02}& {1.26 }& {6.51e-03 }& {1.31 }& { 0.01}\\
   640 &     64 &{3.22e-02} &{ 0.68 }& {5.24e-03}& {1.00 }& {3.38e-03 }& {0.94 }& { 0.03}\\
  1280 &    128 &{1.83e-02} &{ 0.82 }& {2.72e-03}& {0.94 }& {1.76e-03 }& {0.95 }& { 0.11}\\
  2560 &    256 &{9.76e-03} &{ 0.91 }& {1.40e-03}& {0.96 }& {8.95e-04 }& {0.97 }& { 0.55}\\
  5120 &    512 &{5.07e-03} &{ 0.94 }& {7.17e-04}& {0.97 }& {4.53e-04 }& {0.98 }& { 2.36}\\
\hline
\end{tabular}
\end{center}
\caption{\label{tab:CN-2-highCFL} {CN} scheme \eqref{eq:CNscheme-2}
(for solving \eqref{eq:2}) with different mesh parameters $N=J{+1}$ and $N={(}J{+1)}/10$.
}
\end{table}

\section{BDF obstacle schemes} \label{sec:BDFimp}

We now consider BDF type approximations for the first derivative $u_t$, leading to implicit schemes.
We propose two implicit schemes (BDF2 and BDF3) which have the same complexity as the previous CN implicit schemes but give
improved numerical results with respect to precision and to stability.
Furthermore a stability and error analysis will be carried out for the BDF2 scheme.

\subsection{BDF2 obstacle scheme}
Our first scheme is therefore the following two-step implicit scheme (hereafter also referred {to} as BDF2 obstacle scheme),
for $n\geq 1$:
\be
 & & \hspace{-1.0cm} \cH^{n+1}_j(u):\equiv
   \nonumber \\
 & & 
   \min\bigg( \frac{3u^{n+1}_j - 4u^n_j + u^{n-1}_j}{2\dt}  +  (A u^{n+1} + q^{n+1})_j,\ u^{n+1}_j-{\varphi_j^{n+1}}\bigg) - f^{n+1}_j = 0,
   \nonumber \\
 & &
   \label{eq:BDF2scheme}
\ee
initialized with appropriate $u^0$ and $u^1$ values.
Such approximations for the linear term $u_t$, known  as BDF approximations,
are well known and used in various contexts \cite{curtiss52ios,gear71niv}.
For $u^1$, e.\,g.\ the implicit Euler method {(IE) (corresponding to a first order BDF method)}
\[
\min\bigg( \frac{u^{1}_j - u^0_j}{\dt}  +  (A u^{1} + q^{1})_j,\ u^{1}_j-\varphi_j^{1}\bigg) - f^{1}_j = 0,
\]
or the {CN} scheme \eqref{eq:CNscheme-1} with $n=0$ could be used.
In the following, for the numerical tests, the first step $u^1$ is always
computed by a {CN} scheme (see in particular Remark~\ref{rem:technical_details}).

The use of a BDF scheme for a diffusion plus obstacle problem is not new
({Windcliff et al \cite{windcliff01soa},} Oosterlee et al \cite{oos-2003,oos-gas-fri-2003},
the idea was also suggested by Seydel in~\cite[see pages 187 and 217]{seydel-12}).
To the best of our knowledge, a precise analysis of the scheme was missing {so far}.

By construction the scheme has the following consistency error, when $v$ is regular, for $v^n_j=v(t_n,x_j)$:
\be \label{eq:BDF2imp}
  \cH^{n+1}_j(v)
    & = & \min( v_t + \cA v, v-\varphi)(t_{n+1},x_j) - f(t_{n+1},x_j)\nonumber \\
    &   & \ \  + O(\dt^2 \| v_{3t}\|_\infty) + O(\dx^2 (\| v_{3x}\|_\infty + \| v_{4x}\|_\infty)).
\ee
{We summarize this in the following Lemma, to be compared to Lemma~\ref{lem:2.3}.

\begin{lemma}\label{lem:3.2}
If $v$ is regular, the BDF scheme \eqref{eq:BDF2scheme}
is second order consistent in time and space with respect to the obstacle problem~\eqref{eq:1}.
\end{lemma}}%
This consistency error justifies the introduction of BDF schemes that precisely approximate
$u_t + \cA u$ at time $t_{n+1}$ without the need of other particular requirements (there is no requirement that $v_t+\cA v=0$ at previous times
$t<t_{n+1}$, which would not hold in the presence of an obstacle term).

Let $\min(X,Y):=(\min(x_j,y_j))_j$ denote
the minimum of two vectors $X=(x_j),Y=(y_j)$ of $\R^J$.
For convenience, the scheme {\eqref{eq:BDF2scheme}} will also be written as follows:
\begin{multline}
  \min\bigg( (I_J +  \frac{2}{3}\dt A)\, u^{n+1} - \frac{4}{3} u^n + \frac{1}{3} u^{n-1}
           + \frac{2}{3} \dt  q^{n+1} - {\frac{2}{3} }\tau f^{n+1},\ u^{n+1}-{\varphi^{n+1}}-f^{n+1}\bigg) \\  = 0
  \label{eq:uscheme}
\end{multline}
{with $I_J$ denoting the $J$-dimensional identity matrix.} (After subtracting $f^{n+1}$, a multiplication by $\frac{2\dt}{3}>0$ of the left part of the $\min$ term
does not change the equation.)

\begin{rem}[Newton's method]
As already mentioned before, the scheme can be solved by a semi-smooth Newton method {\cite{bok-mar-zid-2009}}.
More precisely, denoting
$B:=I_J+\frac{2}{3}\dt A$ (a real valued $J\times J$ matrix), $\oldbcoeff:= \frac{4}{3} u^n - \frac{1}{3} u^{n-1} - \frac{2}{3} \dt  q^{n+1}{+\frac{2}{3} \dt  f^{n+1}}$,
 {and $g:= \varphi^{n+1}+ f^{n+1}$,}
the problem is to solve for $x\in \R^J$
\be
  \label{eq:obs}
  \min (B x - \oldbcoeff, x- g)=0 \quad \mbox{in $\R^J$.}
\ee
The matrix $B$ satisfies the conditions of Remark {\ref{eq:newton}} ensuring the convergence of Newton's algorithm {provided that $\frac{\dt}{\dx}|b|\leq\frac32+\dt r$}.
\end{rem}

Results for the {BDF2 obstacle} scheme are given in Table~\ref{tab:BDF2-imp} {for $N=J{+1}$ and $N={(}J{+1)}/10$ (larger time steps)}
using the same parameters as {in} \eqref{eq:parameters}.
These results show robustness of the scheme even for large time steps and also an improvement of the convergence with respect
to the CN schemes (the order is closer to $2$ even for $N={(}J{+1)}/10$). {Note that the results indicate second order convergence although by estimate \eqref{eq:BDF2imp} we can only expect second order convergence for solutions that are three times continuously differentiable with respect to time and four times continuously differentiable with respect to space.}

\begin{table}[!hbtp]
\begin{center}
\begin{tabular}{|cc|cc|cc|cc|c|}
\hline
\multicolumn{2}{|c|}{Mesh}
   & \multicolumn{2}{|c|}{Error $L^1$} & \multicolumn{2}{|c|}{Error $L^2$} & \multicolumn{2}{|c|}{Error $L^\infty$} &  time(s)
   \\ \hline
  $J{+1}$ & $N$   &  error   &  order &    error & order  & error    & order   &         \\
\hline\hline
    80 &     80 &{7.11e-01} &{ 1.87 }& {1.26e-01}& {1.79 }& {4.48e-02 }& {1.18 }& { 0.01}\\
   160 &    160 &{1.35e-01} &{ 2.40 }& {2.18e-02}& {2.53 }& {5.09e-03 }& {3.14 }& { 0.01}\\
   320 &    320 &{3.52e-02} &{ 1.94 }& {5.66e-03}& {1.94 }& {1.33e-03 }& {1.93 }& { 0.04}\\
   640 &    640 &{9.12e-03} &{ 1.95 }& {1.46e-03}& {1.96 }& {3.41e-04 }& {1.97 }& { 0.13}\\
  1280 &   1280 &{2.43e-03} &{ 1.91 }& {3.84e-04}& {1.93 }& {8.88e-05 }& {1.94 }& { 0.48}\\
  2560 &   2560 &{6.60e-04} &{ 1.88 }& {1.03e-04}& {1.90 }& {2.72e-05 }& {1.70 }& { 1.56}\\
  5120 &   5120 &{1.73e-04} &{ 1.93 }& {2.61e-05}& {1.98 }& {5.79e-06 }& {2.23 }& { 7.25}\\
\hline\hline
    80 &      8 &{4.56e-01} &{ 1.71 }& {7.75e-02}& {1.49 }& {3.59e-02 }& {0.35 }& { 0.00}\\
   160 &     16 &{7.29e-02} &{ 2.65 }& {9.73e-03}& {2.99 }& {2.20e-03 }& {4.02 }& { 0.00}\\
   320 &     32 &{2.25e-02} &{ 1.69 }& {3.38e-03}& {1.53 }& {8.82e-04 }& {1.32 }& { 0.01}\\
   640 &     64 &{7.31e-03} &{ 1.62 }& {1.17e-03}& {1.53 }& {2.98e-04 }& {1.56 }& { 0.02}\\
  1280 &    128 &{2.17e-03} &{ 1.75 }& {3.52e-04}& {1.74 }& {8.65e-05 }& {1.79 }& { 0.11}\\
  2560 &    256 &{5.22e-04} &{ 2.06 }& {8.21e-05}& {2.10 }& {2.40e-05 }& {1.85 }& { 0.38}\\
  5120 &    512 &{1.07e-04} &{ 2.29 }& {1.39e-05}& {2.56 }& {5.79e-06 }& {2.05 }& { 1.24}\\
\hline
\end{tabular}
\end{center}
\caption{\label{tab:BDF2-imp} BDF2 scheme for \eqref{eq:1}.}
\end{table}

\begin{rem}\label{rem:technical_details}
We have numerically observed that if we compute the first step with an {IE}
obstacle scheme ({corresponding to} BDF1) and the BDF2 scheme is otherwise
unchanged for the next steps,
then the results are not as clear as in Table 3 (second order convergence does not appear clearly),
and {with $N=(J+1)/10$ we observe rather first order convergence}.
\end{rem}

\subsection{BDF3 obstacle scheme}
In the same way,
we propose the following three-step (BDF3) implicit scheme, for $n\geq 2$:
\begin{multline*}
  \cH^{n+1}_j(u) :\equiv
    \min\bigg( \frac{\frac{11}{6} u^{n+1}_j - 3 u^n_j + \frac{3}{2} u^{n-1}_j
       -\frac{1}{3} u^{n-2}_j}{\dt}
      +  (A u^{n+1} + q^{n+1})_j,\\u^{n+1}_j-{\varphi^{n+1}_j}\bigg) - f^{n+1}_j = 0.
\end{multline*}
The scheme may be initialized by any second order approximation for the first two steps $u^1$ and $u^2$,
{we have chosen the CN scheme for $u^1$ and the BDF2 scheme for $u^2$.}

As we have done for the BDF2 scheme, we can multiply the left term by $6\dt$,
define $B:=11\, I_{{J}}  + 6 \dt A$,
and obtain then an equivalent scheme in the following form, for $n\geq 2$, in $\R^J$:
\[
   \min\bigg( B u^{n+1}  - 18 u^n +  9 u^{n-1} - 2 u^{n-2}
      +  6 \dt q^{n+1} -{6}\tau f^{n+1},\ u^{n+1}-{\varphi^{n+1}} - f^{n+1}\bigg) = 0. \]
The unknown $u^{n+1}$ can be solved by using again a
semi-smooth Newton method.

{
We have observed that the numerical results with the BDF3 scheme are not as good as in Table~\ref{tab:BDF2-imp} with the BDF2 scheme
for the American option problem (the order of convergence is two for $N=J{+1}$ and closer to one for $N={(}J{+1)}/10$).
Since there is a jump in the second order derivative $v_{xx}$, we do not expect better than second order convergence in
this case.
We do not have a convergence result for BDF3 since even in the linear case the scheme is known not to be $A$-stable
(see Remark~\ref{rem:BDF3scheme}).
}

The performance of the BDF3 obstacle scheme (using a $4$th order approximation in space)
will be tested in Section~\ref{sec:num:res1-2} on a model problem with
a bounded $v_{3x}$ derivative, showing third order in that case.

\section{Stability and error estimate for the BDF2 scheme}\label{sec:convEstimates}

Throughout this section, we will consider the following assumptions:

\medskip

{\bf Assumption (A1):}
\begin{itemize}
\item $a\equiv \frac{1}{2}\ms^2$, $b$ and $r$ are bounded functions {(this follows already from $\sigma,b,r$ being Lipschitz continuous on the finite domain $\mO$)},
\item there exists $\eta_0>0$ such that $a(t,x)\geq \eta_0>0$ for all $t,x$,
\item $a$ is Lipschitz continuous in $x$ uniformly w.r.t.\ $t$, that is:
\be \label{eq:a-lip}
  \exists L\geq 0, \ |a(t,x)-a(t,y)|\leq L|x-y|, \quad t\in (0,T),\ (x,y)\in \mO^2.
\ee
\end{itemize}

{
\begin{rem}
For the error analysis, no regularity assumption will be needed neither on the obstacle ${\varphi}$ nor the source term $f$. Indeed, these terms
vanish in the consistency error analysis.
\end{rem}

\begin{rem}\label{rem:degenerate-coef}
In the case that the diffusion coefficient may degenerate some analysis may hold without the obstacle term
(see~\cite{Bokanowski_Picarelli_Reisinger_2018}).
\end{rem}

}

\subsection{Stability estimate}\label{sec:stab-estim}

Let us first start by considering an abstract obstacle problem of the form
$$ \min(B \linconnu - \oldbcoeff, \linconnu-g)=0\quad \text{ for }\quad \linconnu\in \R^J,$$
where $B$ is a square matrix of size $J$ and $\oldbcoeff,g$ are given vectors of~$\R^J$.
We will use the following elementary result.

\begin{lemma}\label{lem:Bineq}
For any matrix $B$, the  following equivalence holds:
\be  \min(B\linconnu-\oldbcoeff,\linconnu-g)=0
  & \Leftrightarrow &
    \linconnu\geq g\ \mbox{and}\ \bigg(\< B \linconnu-\oldbcoeff, v-\linconnu\>\geq 0,\ \forall v\geq g\bigg){.}
\ee
\end{lemma}
\proof
It is known~\cite{cia-82}
that if $B$ is a positive definite symmetric matrix,
the following equivalences hold:
\be  \min(B\linconnu-\oldbcoeff,\linconnu-g)=0
  & \Leftrightarrow &   \mbox{$\linconnu$ solves} \  \min_{\linconnu\geq g} \frac{1}{2} \<\linconnu,B\linconnu\> -\<\oldbcoeff,\linconnu\> \label{eq:ineq1} \\
  & \Leftrightarrow &   \linconnu\geq g\ \mbox{and}\ \bigg(\< B \linconnu-\oldbcoeff, v-\linconnu\>\geq 0,\ \forall v\geq g\bigg)
    \label{eq:ineq2}{.}
\ee
When $B$ is not symmetric,
the equivalence between the $\min$ equation and
\eqref{eq:ineq2} is still true:\\
$\Rightarrow$: For $v\geq g$, $\< B\linconnu-\oldbcoeff,v-\linconnu\> = \<B\linconnu-\oldbcoeff,v-g\> + \underbrace{\<B\linconnu-\oldbcoeff,g-\linconnu\>}_{=0}$ so is
{nonnegative} since $B\linconnu-\oldbcoeff\geq0$ and $v-g\geq 0$.\\
$\Leftarrow$: By taking $v=\linconnu + \ml e_j$ with $\ml\conv +\infty$ we get $(B\linconnu-\oldbcoeff)_j\geq 0$, hence
$B\linconnu-\oldbcoeff\geq0$. Then, $\<B\linconnu-\oldbcoeff,\linconnu-g\>\geq 0$, and also $\<B\linconnu-\oldbcoeff,\linconnu-g\>\leq 0$ by taking $v=g$ as a test
function in the inequality. Hence $\<B\linconnu-\oldbcoeff,\linconnu-g\>= 0$. Together with $B\linconnu-\oldbcoeff\geq0$, $\linconnu-g\geq 0$, this implies
that $\min(B\linconnu-\oldbcoeff,\linconnu-g)=0$.
\endproof

The idea now is to use the inequality of Lemma \ref{lem:Bineq} in order to obtain
a stability estimate. For {parabolic problems}, it is possible to obtain stability  estimates in the $L^2$ norm for the Gear (or BDF2) scheme (see for instance~\cite{emm-05}).
We are going to obtain similar estimates for {the scheme \eqref{eq:uscheme} applied to} the obstacle problem {\eqref{eq:1}}.

Let $v(t,x)$ be a regular enough function, $v^n_{{j}}:=v(t_n,x_j)$, and
$\bar\eps^n\in \R^J$ be {defined by
\begin{equation}
\bar\eps^n_j=\frac{1}{2\dt}( 3v^{n+1}_j - 4 v^n_j + v^{n-1}_j ) + (A v^{n+1}
    + q^{n+1})_j
    -(v_t+ \cA v)(t_{n+1},x_j),
 \label{eq:epsn-def}
\end{equation}}
{$n=1,\dots,N-1$.} The term $\bar \eps^n$ corresponds to a consistency error for the linear part of the PDE,
here written in discrete form on the grid mesh.

{
\begin{rem}\label{rem:not-well-defined}
If $v$ is continuous but $v_t,v_x,v_{xx}$ are not well defined at $(t_{n+1},x_j)$, we can still define
$\bar \eps^n_j$ as follows.
We consider a definition of $v_t(t_{n+1},x_j)$, $v_x(t_{n+1},x_j)$ and $v_{xx}(t_{n+1},x_j)$ such that $\bar\eps^n_j$ (defined by \eqref{eq:epsn-def})
satisfies the bound
\be \label{eq:epsn-bound-singular}
  |\bar\eps^n_j|\leq C \bigg( \|v_t(t_{n+1},.)\|_{L^\infty(\mO)} + \|v_x (t_{n+1},.)\|_{L^\infty(\mO)}+ \|v_{xx} (t_{n+1},.)\|_{L^\infty(\mO)}
  \bigg)
\ee
with a constant $C\geq 0$ (independent of $n,j$).
This bound assumes that the exact derivatives $v_t(t_{n+1},.)$, $v_{xx}(t_{n+1},.)$ exist a.e.\ on $\mO$ (with possible discontinuities) and
are bounded, and this will be considered later on in assumption (A2).
For instance, extending the domain of definition of $v_t$, $v_x$ and $v_{xx}$ to whole $\Omega$ by $v_t= v_x=v_{xx}:=0$ at places of non-differentiability is a possible choice.
\end{rem}
}%
Then we have
\be \label{eq:vscheme-0}
  \min\big(\frac{1}{2\dt}( 3v^{n+1} - 4 v^n + v^{n-1} ) + A v^{n+1}
  + q^{n+1} {-f^{n+1}}- \bar\eps^n, \ v^{n+1}-g^{{n+1}})=0
\ee
{with $g^{n+1}:=\varphi^{n+1}+f^{n+1}$.}
Therefore $v^n$ satisfies a perturbed scheme, as follows:
\begin{equation}
  \min \big((I_J+ \frac{2\dt}{3} A) v^{n+1} - \frac{4}{3} v^n + \frac{1}{3} v^{n-1}
   + \frac{2\dt}{3} q^{n+1}{-\frac{2\dt}{3} f^{n+1}} - \frac{2\dt}{3} \bar\eps^n,\
   v^{n+1}- g^{{n+1}}\big) =   0.
   \label{eq:vscheme}
\end{equation}

\begin{rem}
Typically $\bar\eps^n$ is of order $O(\dt^2 + h^2)$ where $v$ is
regular.
\end{rem}

Our aim is now to show a stability estimate in order to control the error
$\|u^n-v^n\|^2_2$ in terms of $\sum_{1\leq k\leq n-1} \dt \|\bar \eps^n\|^2$.

For a vector $x=(x_j)_{1\leq j\leq J}$, let
\be\label{eq:Nseminorm}
  N(x):= \left(\sum_{j=1}^{J+1} |x_j - x_{j-1}|^2 \right)^{1/2}
\ee
(with the convention $x_0:=0$ and $x_{J+1}:=0$).

{
The following shows a coercivity bound for the matrix $A$.
}

\begin{lemma}\label{lem:4.6}
Under assumption (A1),
there exist $\eta>0$ and $\mg\geq 0$ such that
{for $A$ given by \eqref{eq:A2ndorder} and}
for all $e\in\R^J$:
\be\label{eq:coer}
  \<e,Ae\> \geq \eta N(e/h)^2 - \mg \|e\|_2^2.
\ee
\end{lemma}
\begin{proof}
Considering $\cA u = - a(t,x) u_{xx} + b(t,x) u_x + r(t,x) u$
and hereafter {not} explicitly mentioning the time variable,
it holds
$$
  A=\frac{1}{h^2}\tridiag(-a_i,\ 2 a_i,\ -a_i) + \frac{1}{2h}\tridiag(-b_{i},\ 0,\ b_{i}) + \diag({r}_i)
$$
where $a_i=a(x_i)$, $b_i=b(x_i)$ and $r_i=r(x_i)$.  By straightforward calculations,
\begin{equation}
  \<e,Ae\>
  =
    \frac{1}{h^2} \sum_{i=1}^{J+1} (a_i e_i - a_{i-1} e_{i-1})(e_i-e_{i-1})
   + {\frac1{2h}}\sum_{i=1}^J b_i (e_{i+1}-e_{i-1}) e_i + \sum_{i=1}^J {r}_i e_i^2.
\end{equation}
Now we make use of $|a_{i}-a_{i-1}|\leq C h$ for some constant $C\geq 0$ (since $a(\cdot)$ is Lipschitz continuous), and $a_i\geq \eta_0$, to obtain:
\begin{align}
  \frac{1}{h^2} \sum_{i=1}^{J+1} (a_i e_i - a_{i-1} e_{i-1})(e_i-e_{i-1})
  & \geq
  \frac{1}{h^2} \sum_{i=1}^{J+1} \eta_0 (e_i-e_{i-1})^2 - {\frac{1}{h}}\sum_{i=1}^J C |e_{i}|\,|e_i - e_{i-1}| \nonumber \\
  & \geq
  \eta_0 N(e/h)^2 - C \|e\|_2 N(e/h).
\end{align}
We have also, by using $e_{i+1}-e_{i-1}=(e_{i+1}-e_i) + (e_i - e_{i-1})$:
\[
  \sum_{i=1}^J |b_i (e_{i+1}-e_{i-1}) e_i | \leq \|b\|_\infty 2 N(e) \|e\|_2.
\]
Hence {there exists} a lower bound {of the} form: 
\beno
  \<e,Ae\>  \geq \eta_0 N(e/h)^2 - (C + \|b\|_\infty) N(e/h) \|e\|_2 - C \|e\|_2^2
\eeno
for some constant $C$. Denoting $C':=C + \|b\|_\infty$ and
applying the inequality $C' N(e/h) \|e\|_2 \leq \frac{\eta_0}{2} N(e/h)^2 + \frac{C'^2}{2\eta_0} \|e\|_2^2$,
we finally obtain
\beno
  \<e,Ae\>  \geq \frac{\eta_0}{2} N(e/h)^2 - (C+ \frac{C'^2}{2\eta_0}) \|e\|_2^2
\eeno
which gives the desired lower bound with $\eta=\eta_0/2$ and $\mg=C+\frac{C'^2}{2\eta_0}$.
\end{proof}

From now on we shall denote the error by
$$
  e^n:=v^n- u^n.
$$

\begin{prop}\label{prop:stab-estim}
Consider the scheme \eqref{eq:uscheme},
and a perturbed scheme \eqref{eq:vscheme}. 
Let $\dt>0$ be sufficiently small.
Then there exist a constant $C_1$ independent of $n$
and a constant $\bar \mg >0$
such that for all $t_n\leq T$

\begin{multline}
  e^{-\bar\mg t_n} \|e^{n{+1}}\|^2_2 + \dt{\eta} \sum_{k=1}^{n} e^{-\bar \mg t_k}  N(e^{k{+1}}/h)^2\\\leq\ C_1\bigg(\|e^0\|^2_2 + \|e^1\|^2_2 + \dt \sum_{k=1,\dots,n} e^{-\bar \mg t_k} \|\bar\eps^{{k}}\|^2_2\bigg)\label{eq:stab-estim-bis}
\end{multline}
where {$N(e^k/h)$ is defined by \eqref{eq:Nseminorm}}.
\end{prop}

\begin{proof}[Proof of Proposition~\ref{prop:stab-estim}]
Let
$$ B := I_{{J}}+ \frac{2\dt}{3} A,$$
and vectors $b_u$, $b_v$ be such that
$$
  b_u := \frac{4}{3} u^n - \frac{1}{3} u^{n-1} - \frac{2\dt}{3} q^{n+1}{+\frac{2\dt}{3} f^{n+1}}
$$
and
\[
  b_v := \frac{4}{3} v^n - \frac{1}{3} v^{n-1} - \frac{2\dt}{3} q^{n+1}{+\frac{2\dt}{3} f^{n+1}}.
\]
Then, by Lemma~\ref{lem:Bineq}, the $\min$ equation \eqref{eq:uscheme} for the exact scheme is equivalent to $u^{n+1}\geq g^{{n+1}}$ and
\begin{equation}\label{eq:1var}  \big< B u^{n+1} - b_u,\ w - u^{n+1}\big> \geq 0, \quad \forall w \geq g^{{n+1}}.
\end{equation}
The $\min$ equation \eqref{eq:vscheme} for the perturbed scheme is equivalent to $v^{n+1}\geq g^{{n+1}}$ and
\begin{equation}\label{eq:2var}  \big< B v^{n+1} - (b_v+{\frac23}\dt\bar\eps^n),\ w - v^{n+1} \big> \geq 0,
  \quad \forall w \geq g^{{n+1}}.
\end{equation}
Taking $w=v^{n+1}$ in \eqref{eq:1var} gives
\[ \big< B u^{n+1} - b_u ,\ v^{n+1} - u^{n+1} \big> \geq 0, \]
and $w=u^{n+1}$ in \eqref{eq:2var} gives
$$ \big< B v^{n+1} - (b_v+{\frac23}\dt \bar\eps^n) ,\ u^{n+1} - v^{n+1} \big> \geq 0. $$
Combining the last two relations gives
\be
   \big< B e^{n+1} - \frac{4}{3} e^n + \frac{1}{3} e^{n-1} - \frac{2\dt}{3}\bar\eps^n ,\ e^{n+1}\big> \leq  0
\ee
and therefore
\be \label{eq:firstineq}
   \big< 3 e^{n+1} - 4 e^n + e^{n-1},\ e^{n+1}\big> + 2 \dt \<e^{n+1}, A e^{n+1}\> \leq  2 \dt \<\bar\eps^n,\,e^{n+1}\> .
\ee
Now let $x_n$, $y_n$ and $z_n$ be defined by
$$
  x_n := \| e^n\|_2^2,
  \qquad y_n:=\| e^{n+1}-e^n\|_2^2,
  \qquad z_n:=2\dt \<Ae^{n{+1}},\,e^{n{+1}}\>.
$$
The following estimate holds:
\be \label{eq:err1}
  3 x_{n+1} - 4 x_n + x_{n-1} + 2 y_n + {2}z_n
  \leq  2 y_{n-1} + {4}\dt \<\bar\eps^n,e^{n+1}\>.
\ee
To prove \eqref{eq:err1}, we first use
the properties
$\<a -b,\,a\> = \frac{1}{2} (\|a\|_2^2 + \|a-b\|_2^2 - \|b||_2^2)$
as well as $\frac{1}{2}\|a+b\|_2^2 \leq  \|a\|_2^2 + \|b\|_2^2$,
to obtain
\beno
  && \hspace{-10mm}2 \<3e^{n+1}  - 4 e^n + e^{n-1},\, e^{n+1}\>\\
  & = &  2(4 \<e^{n+1}-e^n,\,e^{n+1}\> - \<e^{n+1}-e^{n-1},\, e^{n+1}\>)\\
  & = &  4(x_{n+1}+y_{n}-x_n) - (x_{n+1} + \|e^{n+1}-e^{n-1}\|_2^2 - x_{n-1})\\
  & \geq &  4(x_{n+1}+y_{n}-x_n) - (x_{n+1} + 2(y_n+ y_{n-1})- x_{n-1}) \\
  & \geq &  3 x_{n+1}-4x_n + x_{n-1} + 2y_n -2y_{n-1}
\eeno
and we conclude by using \eqref{eq:firstineq}.

Let
$$ w_n: = 4\dt \eta N(e^{n+1}/h)^2. $$
By using the bound $2\dt \<\bar\eps^n,e^{n+1}\>
   \leq 2 \dt \|\bar\eps^n\|_{{2}} \|e^{n+1}\|_{{2}}
   \leq \dt \|\bar\eps^n\|_2^2 +  \dt x_{n+1}$
and the coercivity \eqref{eq:coer},
we obtain, for $n\geq 1$:
\be \label{eq:err2}
   3 x_{n+1}-4x_n + x_{n-1} + 2y_n  +  w_n 
    \leq  2y_{n-1}  + 2\dt \|\bar \eps^n\|_2^2 + (2 \dt + 4 \dt \mg)  x_{n+1}.
\ee

Let
$$
  \bar \mg:= 2 + 4 \mg \quad \mbox{ and } \quad \mb:= \dt \bar \mg.
$$
It follows
\be \label{eq:err2b}
   (3- \mb) x_{k+1}-4x_k + x_{k-1} + 2y_k  +  w_k 
    \leq  2y_{k-1}  + 2\dt \|\bar \eps^k\|_2^2, \quad k\geq 1.
\ee

We multiply \eqref{eq:err2b} by $e^{-k\mb}$ and sum up the inequalities from $k=1$ to $n\geq 1$.
Let $f(x)=x^2  - 4x + 3-\mb$ and notice that for $\tau$ small enough (and therefore small $\mb>0$), $f(e^{-\mb})\sim\mb>0$. We deduce that for some constant $C_{01}$
\be
  & & \hspace{-1cm} e^{-n \mb} ( (3-\mb) x_{n+1} - (4-(3-\mb)e^{\mb}) x_n) + \sum_{k=1}^n e^{-k\mb} w_k \nonumber \\
  & & \hspace{-1cm} + \sum_{k={2}}^{n-{1}} e^{-(k-1)\mb}  f(e^{-\mb}) x_{k}
       + 2\sum_{k=1}^n e^{-k\mb} y_k \nonumber \\
  & & \leq\ C_{01}(x_0 + x_1)
         + 2 \sum_{k=1}^n e^{-k\mb} y_{k-1}
         + 2 \sum_{k=1}^n e^{-k\mb}{\dt} \|\bar \eps^k\|_2^2.
\ee
Using that $e^{-k\mb} y_{k-1} \leq e^{-(k-1)\mb} y_{k-1}$ and $f(e^{-\mb})>0$ we deduce that
\be
  & & \hspace{-1cm} e^{-n \mb} ( (3-\mb) x_{n+1} - (4-(3-\mb)e^{\mb}) x_n) + \sum_{k=1}^n e^{-k\mb} w_k \nonumber \\
  & & \leq\ C_{01}(x_0 + x_1) + 2 e^{-\mb} y_{0} + \sum_{k=1}^n e^{-k\mb}  2\dt \|\bar \eps^k\|_2^2 \nonumber \\
  & & \leq\ C_{02}(x_0 + x_1 +  \dt \sum_{k=1}^n e^{-k\mb}  \|\bar \eps^k\|_2^2) =: Q  \label{eq:lastineq1}
\ee
with $C_{02}:=C_{01}+4$ (where we have used that $y_0\leq 2(x_0+x_1)$).

Let us prove that $e^{-\bar\mg t_n} x_n\leq x_1 + C_{02} Q$, which will give the desired bound.
By using $k\mb{=}k\tau \bar\mg{=}\bar \mg t_k$,
we deduce from \eqref{eq:lastineq1}
$$
  x_{k+1} \leq \rho x_k + \frac{e^{\bar \mg t_n} Q}{{3-\beta}}, \quad 1\leq k\leq n,
$$
where $\rho:= \frac{4-(3-\mb)e^\mb}{3-\mb} \sim \frac{1}{3}$ as $\mb=\dt \bar \mg \conv 0$.
By recursion we get for $1\leq k\leq n$:
\beno
  x_{k}  & \leq &  \rho^{k-1} x_{1} + \frac{e^{\bar \mg t_n} Q}{{3-\beta}} (1+ \rho + \cdots + \rho^{k-2}) \\
    & \leq &  x_{1} + \frac{e^{\bar \mg t_n} Q}{{3-\beta}} \frac{1}{1-\rho}.
\eeno
By using this bound for $x_n$ into \eqref{eq:lastineq1},
we obtain the desired result
\eqref{eq:stab-estim-bis}
with a possibly different universal constant $C_1$.
This concludes the proof of Proposition~\ref{prop:stab-estim}.
\end{proof}

\begin{rem}[BDF3 scheme] \label{rem:BDF3scheme}
The previous stability estimate does not extend easily to the BDF3 obstacle scheme.
Indeed, it is known that BDF3 is not $A$-stable (as well as any BDF$k$ for $k\geq 3$, see \cite{hai-wan-10,Hairer_Wanner_96}),
which prevents the same stability analysis to apply for diffusion equations.
\end{rem}

\subsection{Error estimate for the BDF2 scheme}\label{sec:error-estim}

The following assumptions will be used.

\medskip

{\bf Assumption (A2).}
We assume that there exist an integer $p\geq 0$ and continuous functions $t\conv y_i(t)$ for $i=1,\dots,p$ with
$y_i(t)\in [X_{min},X_{max}]$ such that,
{
defining for $0\leq \tau<T$
$$\mO_{\tau,T}:=\{(t,x)\in (\tau,T)\times\mO,\ x \notin (y_i(t))_{1\leq i\leq p}\}$$
}
($\mO_{\tau,T}$ is a subdomain of $(0,T)\times\mO$), the following holds:

\begin{itemize}
\item{$(i)$}
{$(t,x)\conv v(t,x)$  is regular (i.\,e.\ $C^{2,3}$) on $\mO_{0,T}$,}

\item{$(ii)$}
{
there exist constants $\ma_i\geq 0$, $C_i\geq 0$, $i=1,\dots,4$, such that for all $\eps>0$:
\begin{gather*}
  \|v_{t}\|_{L^\infty(\mO_{\eps,T})}\leq C_{1}\eps^{-\ma_1}, \quad \|v_{{xx}}\|_{L^\infty(\mO_{\eps,T})}\leq C_{3}\eps^{-\ma_3},\\
  \|v_{{tt}}\|_{L^\infty(\mO_{\eps,T})}\leq C_2 \eps^{-\ma_2}, \quad \|v_{{xxx}}\|_{L^\infty(\mO_{\eps,T})}\leq C_{4} \eps^{-\ma_4}.
\end{gather*}
}
\end{itemize}

\begin{rem}
Assumption (A2) allows for $v_{xx}(t,.)$ to have \enquote{jumps} at the singular points $x=y_k(t)$.
\end{rem}

\medskip

{\bf Assumption (A3).} There exists {$\ma_0\in(0,1]$} such that,
for $i{=1,\dots,p}$, $t\conv y_i(t)$ is $\ma_0$-H{\"o}lder continuous on $[0,T]$.

\medskip

{
If the first step of the BDF2 scheme is initialized with the CN scheme, the following assumption will also be needed:

\medskip

{\bf Assumption (A4).}  $v_0$ is Lipschitz continuous and piecewise $C^2$ regular on $\mO$.
}

\medskip

{
Explicit examples satisfying assumptions (A2)--(A4)
will be given in the numerical section, see Remark \ref{rem:thm-error-estim_applied_to_Model1}.
}

\begin{rem} \label{rem:xs-amer}
For the American put option problem \eqref{eq:amer-pb}
with $\varphi(x):=(K-x)_{+}$ it is known that there is a unique singular point $y_1(t)\equiv x_s(t)$ such that
$v(t,x)=\varphi(x)$ for $x<x_s(t)$, $v(t,x)>\varphi(x)$ for $x>x_s(t)$, and that
\be \label{eq:rapid}
  1-x_s(t)/K \stackrel{t\conv 0^+}{\sim}  \ml (t |\ln(t)|)^{1/2}
\ee
(see \cite{bar-bur-rom-sam-1993}, and \cite[Chap 6]{ach-pir-2005}, as well as \cite{dew-how-rup-wil-1993}).
Furthermore, a function of the form of $(t |\ln(t)|)^{1/2}$ satisfies assumption (A3) for any $\ma<1/2$.
{We do not know if the American option problem satisfies (A2).}
Nevertheless, assumptions (A2)-(A3) allow for closely related problems where $\varphi$ is {Lipschitz continuous} and piecewise $C^2$
and by allowing some rapidly moving singularities $y_i(t)$ as $t\conv 0$ (such as \eqref{eq:rapid}).
\end{rem}

{
In the following error analysis, we consider the continuous $L^2$ norm on $\mO$, $\|f\|_{L^2(\mO)}:= (\int_{\mO} |f(x)|^2 dx)^{1/2}$.
We denote by $u^n$ and $\bar v^n$ the following piecewise constant functions of $L^2(\mO)$:
\beno
   u^n(x)  & := & \sum_j u^n_j 1_{I_j}(x), \\
  \bar v^n(x) & := & \sum_j v^n_j 1_{I_j}(x)= \sum_j v(t_n,x_j) 1_{I_j}(x),
\eeno
where $I_j=(x_{j}-h/2,x_j+h/2)$ and $1_{I_j}(x)=1$ if $x\in I_j$ and $1_{I_j}(x)=0$ otherwise,
and we denote also the corresponding error  $\bar e^n := u^n - \bar v^n$.
Our aim is therefore to bound the following continuous $L^2$ error
\be\label{eq:error-equality}
  \|\bar{e}^n\|_{L^2(\mO)} = \bigg(\sum_{i} h |e^n_i|^2 \bigg)^{1/2}.
\ee

\begin{rem}\label{rem:projection}
Notice that there is a uniform Lipschitz bound for $\|v_x(t_n,.)\|_{L^\infty}$
(for instance by using the representation formula \eqref{eq:stop-time-pb-phi-f} for the obstacle problem),
therefore the error introduced by the projection on piecewise constant functions is roughly bounded by $\|\bar v^n - v^n \|_{L^2(\mO)} \leq C h$.
This projection error will not be considered hereafter.
\end{rem}
}

{
\begin{theorem}{\bf (error estimate).} \label{thm:error_estim}
Assume that the exact solution $v$ of \eqref{eq:1} satisfies (A1), (A2) and (A3).
{We consider the BDF2 scheme initialized with an {IE} or a {CN} step for $u^1$.
In the second case furthermore (A4) is assumed.}
Then the BDF2 scheme satisfies the following error bound {for sufficiently small $\tau$ and $h$}:
\begin{multline}
   \max_{1\leq n\leq N} \| \bar{e}^n \|_{L^2(\mO)}^{{2}}
  \leq
    C \big( h^2 \tau^{1-2\ma_4} + h \tau^{1-2\ma_3}
  + \tau^2 \tau^{1-2\ma_2}
  + (\tau^{\ma_0} + h) \tau^{1-2\ma_1}
  +   \tau^{\ma_0}  + h \\+ {\frac{\tau^2}{h}} \big)
  \label{eq:thm-error-estim}
\end{multline}
for some constant $C\geq 0$ independent of $(\tau,h)$
if the powers $(\mb_i)_{1\leq i\leq 4}:=\{(1-2\ma_i)_{1\leq i\leq 4}$
are all non-zero (otherwise any $\tau^{\mb_i}$ with $\mb_i=0$ should be replaced by $\ln(\tau)$), {and $\alpha_2<1$}.

The term {$\frac{\tau^2}{h}$} is not needed if $u^1$ is initialized with IE.

\end{theorem}
}
{
In particular the scheme is convergent as soon as all factors in \eqref{eq:thm-error-estim} converge to $0$ as $(\tau,h)\conv 0$.

\begin{rem}\label{rem:thm-error-estim_applied_to_Model1}
In the case of the Model 1 presented in Section~\ref{sec:numericalresults},
for some $\ma\in(0,1)$ we will have $\ma_0=\ma$,
$\ma_1=1-\ma$,
$\ma_2=2-\ma$, $\ma_3=\ma $, $\ma_4=2\ma$.
Taking furthermore $\tau\equiv h$, the error estimate~\eqref{eq:thm-error-estim} is of order
\beno
  & & \hspace{-1cm} h^{3-2\ma_4} + h^{2-2\ma_3} + h^{3-2\ma_2} + h^{{1+\ma_0}-2\ma_1} +  h^{\ma_{{0}}} + {h} \\
  & \leq & C_1( h^{3-4\ma} + h^{2-2\ma} + h^{2\ma-1} + h^{\ma} + {h})\\
  & \leq & C_2( h^{{\min(3-4\ma,2-2\ma,2\ma-1)}})
\eeno
and does only give convergence for $\ma$ in $(\frac{1}{2},\frac{3}{4})$, with {a square} error estimate of order $O(h^{\min(3-4\ma,2\ma-1)})$.
\end{rem}
}

\begin{proof}[Proof of Theorem \ref{thm:error_estim}]

Let us first consider the approximation in the $x$ variable. {For $n=0,\dots,N-1$, let}
$\bar \eps^{n,1}$ be a consistency error in space, defined by
\be\label{eq:epsn1}
  {
  (A v^{n+1} + q^{n+1})_i  = (\cA v)(t_{n+1},x_i) + \bar\eps^{n,1}_i.
  }
\ee
If $(t_{n+1},x_i)$ corresponds to a singular point of $v$, we consider for $\cA v(t_{n+1},x_i)$
definitions of $v_t$ and $v_{xx}$ that satisfy the
bounds $|v_t|\leq \|v_t\|_{L^\infty}$ and $|v_{xx}|\leq \|v_{xx}\|_{L^\infty}${, see Remark~\ref{rem:not-well-defined}.}
In the region where $x\conv v(t_{n+1},x)$ is regular,
assuming that $v_{3x}(t_{n+1},.)$ is bounded on the interval
$[x_{i-1},x_{i+1}]$, by using Taylor expansions up to the $3$-{rd} order derivatives, it holds
$$
  |\bar\eps^{n,1}_i|= C_{4}t_{n+1}^{-\ma_4}\, O(h).
$$
On the contrary in a region $[x_{i-1},x_{i+1}]$ that may encounter a singularity $y_j(t)$,
we have no more than a bounded second order derivative ($v_{xx} \in L^\infty$).
By using
{$|(A v^{n+1} + q^{n+1})_i|\leq C(\|v_{xx}(t_{n+1},.)\|_{L^\infty}+\|v_{x}(t_{n+1},.)\|_{L^\infty}+\|v(t_{n+1},.)\|_{L^\infty})$},
we have
{
$$
  |\bar\eps^{n,1}_i |= C_{3} t_{n+1}^{-\ma_3} O(1).
$$
}
Moreover,
\begin{equation}\label{eq:cardsingindices}
   \mathrm{Card}\bigg\{i,\ [x_{i-1},x_{i+1}]\cap \{y_j(t_{n+1})\}_{1\leq j \leq p} \neq \emptyset \bigg\}\leq 3 p.
\end{equation}
{
Using that the number of regular terms is bounded by ${J}\leq C/h$,
we obtain
\be
  \|\bar \eps^{n,1}\|^2
   \, = \, \sum_i |\bar \eps^{n,1}_i|^2
 & =    & \sum_{i, regular} |\bar \eps^{n,1}_i|^2  +  \sum_{i, singular} |\bar \eps^{n,1}_i|^2  \\
 & \leq &  C \frac{1}{h} (C_{4} t_{n+1} ^{-\ma_4} h)^2  + C 3p\, (C_{3} t_{n+1}^{-\ma_3}  )^2 \\
 & \leq &  C h t_{n+1}^{-2\ma_4}   + C t_{n+1}^{-2\ma_3}\label{eq:epsbarnkomma1}
\ee
for some constant $C$.

Notice that for any $n\tau \leq T$, and $\tau$ sufficiently small, we have
\be
    & &  \tau \sum_{k=1}^n \frac{1}{t_k^\mb}
      \leq \bigg\{\barr{l}
        C\, \max(\tau^{1-\mb},1)\quad \mbox{for $\mb>0$, $\mb\neq 1$},      \\
        C\, |\ln(\tau)|\quad \mbox{for $\mb=1$},
      \earr
	\label{eq:sumtk}
\ee
where $C$ may depend on $T,\mb$ but is independent of $\tau,n$.

Therefore we obtain, for
$\ma_3,\ma_4\neq \frac{1}{2}$:
\be
  h \big( \tau \sum_{k=1}^{n-1} \|\bar \eps^{k,1}\|^2 )
    & \leq & C \bigg( h^2 \max(\tau^{1-2\ma_4},1)+ h \max(\tau^{1-2\ma_3},1) \bigg) \\
    & \leq & C \bigg( h^2 \tau^{1-2\ma_4} + h \tau^{1-2\ma_3}  + h \bigg) \label{eq:bbb1}
\ee
{(if $\ma_3=\frac{1}{2}$ or $\ma_4=\frac{1}{2}$ then the corresponding
term $\tau^{1-2\ma_i}$ should be replaced by $\ln(\tau)$).}
}

Now we consider the approximation {by BDF2} in time. Let $\bar\eps^{n,2}_i$ be such that, for $n\geq 1$:
\be\label{eq:epsn2}
  \frac{3 v^{n+1}_i - 4 v^n_i + v^{n-1}_{i}}{2\dt} = v_t(t_{n+1},x_i) + \bar \eps^{n,2}_i.
\ee
{
If $t\conv v(t,x_i)$ is regular on $[t_{n-1},t_{n+1}]$ with bounded $v_{{tt}}$ derivative, elementary Taylor expansions and (A2) give,
for $n\geq 2$ (the cases $n=0$ and $n=1$ will be treated separately):
$$
  |\bar \eps^{n,2}_i|\leq C \max_{[t_{n-1},t_{n+1}]}\|v_{{tt}}\|_{L^\infty(\mO)}\,\tau \leq C t_{n-1}^{-\ma_2} \tau,
$$
while otherwise in a singular region we have
$$
  |\bar \eps^{n,2}_i|\leq C \max_{[t_{n-1},t_{n+1}]}\|v_{t}\|_{L^\infty(\mO)}  \leq C t_{n-1}^{-\ma_1}.
$$
}

Let us introduce a set of singular indices as follows:
\be\label{eq:Js}
   {\cI^n_s}:=\bigg\{ i,\ x_i \in \bigcup_{j=1,\dots,p} y_j([t_{n-1},t_{n+1}])\bigg\}{,\qquad n=1,\dots,N-1}.
\ee
For $n\geq {1}$ and for $t\in \Theta_n:=[t_{n-1},t_{n+1}]$,
we get
$$
  |y_j(t)-y_j(t_{n+1})|\leq C(2\tau)^{\alpha_0}.
$$
So if $x_i\in y_j(\Theta_n)$ then $|x_i-y_j(t_{n+1})|\leq  C (2\tau)^{\ma_0}$.
Then, for any $A>0$,
the number of integers $i$ such that $x_i \in [-A+c,A+c]$ is bounded by $2A/h+1$.
Hence, for $n\geq {1}$,
we deduce a bound in the form
\be\label{eq:boundIs}
  \mathrm{Card}(\cI^n_s)\leq C \left(\frac{\tau^{\ma_0}}{h}  + 1\right)
\ee
for some constant $C\geq 0$.
{This bound holds also for \[\cI^0_s:=\bigg\{ i,\ x_i \in \bigcup_{j=1,\dots,p} y_j([t_{0},t_{1}])\bigg\},\] as $\cI^0_s\subset\cI^1_s$.}

Now we can bound the $\bar{\eps}^{n,2}$ terms, for $n\geq 2$, as follows. We have
{
\be
  \|\bar \eps^{n,2}\|^2
 & =    & \sum_{i \notin \cI^n_s} |\bar \eps^{n,2}_i|^2  +  \sum_{i\in \cI^n_s} |\bar \eps^{n,2}_i|^2  \\
 & \leq & C \sum_{i \notin \cI^n_s} (t_{n-1}^{-\ma_2}\,\tau)^2  +  \sum_{i\in \cI^n_s} (t_{n-1}^{-\ma_1})^2 \\
 & \leq & C \frac{1}{h} \tau^2\ t_{n-1}^{-2\ma_2}  + C ( \frac{\tau^{{\ma_0}}}{h} + 1 ) t_{n-1}^{-2\ma_1}.
\ee

Combining the previous bounds and using \eqref{eq:sumtk}, for $\ma_1,\ma_2\neq \frac{1}{2}$,
we obtain
\be
   h \big( \tau \sum_{k=2}^n \|\bar \eps^{k,2}\|^2 \big)
  & \leq & C \big( \tau^2\ \max(\tau^{1-2\ma_2},1) + (\tau^{\ma_0}+h) \max(\tau^{1-2\ma_1},1)\nonumber \\
  & \leq & C \big( \tau^{2}\tau^{1-2\ma_2} + (\tau^{\ma_0}+h) \tau^{1-2\ma_1} + \tau^{\ma_0} + h\big)
    \label{eq:bbb2}
\ee
{(powers of $\tau$ with exponent $1-2\ma_1=0$ or $1-2\ma_2=0$ need to be replaced by $\ln(\tau)$).}
}

Using the stability estimate \eqref{eq:stab-estim-bis} and the fact that $e^0=0$ {we obtain}
{
\be
  \|{\bar e}^n\|_{L^2(\mO)}^2
    & {=}   & h \|e^n\|^2  \\
    & \leq & C h \bigg( \|e^1\|^2 + \tau \sum_{k=1}^{n-1} \|\bar{\eps}^{k,1}\|^2  + \tau \sum_{k=1}^{n-1} \|\bar{\eps}^{k,2}\|^2 \bigg).
\ee
By using the estimates \eqref{eq:bbb1} and \eqref{eq:bbb2}, we obtain the desired error bound \eqref{eq:thm-error-estim} as long as we can bound
$\|\bar{\eps}^{1,2}\|^2$ (the first time-consistency error term that appears in the estimates for BDF2) as well as
$\|e^1\|^2$ (the {IE} {rsp.\ {CN}} scheme error) accordingly.
}

By using similar techniques as for BDF2, {and $e^0=0$,} it is easy to see that
\be
  h\|e^1\|^2 \leq {h\left(\frac{1}{1-s\tau\gamma} (\|e^0\| + \tau \|\bar{\eps}^0\|)\right)^2}
   =  {\frac{h\tau^2}{(1-\tau s\gamma)^2} \|\bar{\eps}^0\|^2 },
\ee
where $\bar{\eps}^0$ is the consistency error for the {IE} {(rsp.\ {CN})} scheme {and $s=1$ (rsp.\ $s=\frac12$)}.

{For both the {IE} and the {CN} scheme, we can write }$\bar{\eps}^0=\bar{\eps}^{0,1} + \bar{\eps}^{0,2}$ {where
$\bar{\eps}^{0,1}$ represents the spatial consistency error and $\bar{\eps}^{0,2}$ the time-consistency error given by
\be\label{eq:epsn2-0}
  \frac{v^{1}_i - v^0_i}{\dt} = v_t(t_{1},x_i) + \bar \eps^{0,2}_i.
\ee
}
The term $|v_t(t,x)|$ is not assumed to be bounded for $t=0^+$, but we have $|v_t(t_1,x)|\leq C \tau^{-\ma_1}$ since $t_1=\tau$ and
using (A2).
By using again (A2) we {obtain}
\begin{equation}\label{eq:special-estim}
  |\frac{1}{\tau}(v(t_{1},x_i)-v(t_0,x_i))-v_t(t_1,x_i)|
    =  \frac{1}{\tau}  |\!\int_{t_0}^{t_1} v_t(s,x_i)\ds|
    \leq \frac{1}{\tau} \int_{t_0}^{t_1} C s^{-\ma_1}\ds
    \leq C \tau^{-\ma_1}.
\end{equation}
This estimate holds for all $i$.
Therefore
$$
  |\bar{\eps}^{0,2}_i| \leq C \tau^{-\ma_1}.
$$
{
If $t\conv v(t,x_i)$ is regular on $[t_{0},t_{1}]$ with bounded $v_{{tt}}$ derivative, the estimate can be improved to
\begin{multline*}
   |\bar{\eps}^{0,2}_i| \leq|\bar{\eps}^{0,2}_i| |\frac{1}{\tau}(v(t_{1},x_i)-v(t_0,x_i))-v_t(t_1,x_i)-v_t(t_1,x_i)|\\
    =  \frac{1}{\tau}  |\int_{t_0}^{t_1}\int_{s}^{t_1}v_{{tt}}(u,x_i)\du\ds|
    \leq \frac{1}{\tau} \int_{t_0}^{t_1} C_2 s^{-\ma_2}(t_1-s)\ds
    \leq C \tau^{1-\ma_2}
\end{multline*}
as $\ma_2<1$.
}
In the end, we obtain a contribution to the error as follows {(to be multiplied by $\tau$)}:
{
\[
  h\tau\|\bar \eps^{0,2}\|^2
 =  h\tau\sum_{i \notin \cI^n_s} |\bar \eps^{0,2}_i|^2  +  h\tau\sum_{i\in \cI^n_s} |\bar \eps^{0,2}_i|^2
 \leq C \tau^{3-2\ma_2}  + C ( \tau^{\ma_0} + h) \tau^{1-2\ma_1}.
\]
}%
The same bound for $h\tau\|\bar{\eps}^{1,2}\|^2$ can be obtained by using similar estimates.

{For the {IE} scheme, the consistency error in space $\bar{\eps}^{0,1}$ can be bounded by \eqref{eq:epsbarnkomma1} with $n=0$, and in} conclusion we obtain the desired bound {for the {IE} scheme as starting scheme.}

{Now, it remains to bound the {spatial} consistency error $\bar{\eps}^{0,1}$ in the case that $u^1$ is computed by a {CN} scheme.
{We split $\bar{\eps}^{0,1}$ into two parts, $\bar{\eps}^{0,1}=\frac12\bar{\eps}^{0,1,1}+\frac12\bar{\eps}^{0,1,2}$, with
\be\label{eq:epsn1-v0}
  (A v_0 + q^0)_i  = (\cA v)(t_1,x_i) + \bar\eps^{0,1,1}_i,\qquad (A v^1 + q^1)_i  = (\cA v)(t_1,x_i) + \bar\eps^{0,1,2}_i.
\ee
$\bar\eps^{0,1,2}$ can be estimated by \eqref{eq:epsbarnkomma1} with $n=0$. To bound $\bar\eps^{0,1,1}$, we take advantage of assumption (A4) being true. Due to $v_0$ being Lipschitz regular, $A v_0 + q^0$ is bounded by $O(\frac{1}{h})$. Since $\cA v_0$ is also assumed to be bounded,
it results a bound of the form
$$
  |\bar\eps^{0,1,1}_i|\leq \frac{C}{h}.
$$}
If $v_0$ is $C^2$ regular on  $[x_{i-1},x_{i+1}]$, then, by standard estimates,
$$
{|(A v_0 + q^0)_i|\leq C \left(\|v''_0\|_{L^\infty}+\|v'_0\|_{L^\infty}+\|v_0\|_{L^\infty}\right)}.
$$
{This, together with $v_{{xx}}$ being bounded,} shows that $|\bar\eps^{0,1,1}_i|$ is bounded for such indices.
Summing up the estimates in the singular region (which {by \eqref{eq:cardsingindices}} involves {not more than $3p$} cases), and in the regular region (which involves
$O(\frac{1}{h})$ cases),
we obtain the bound
\beno
  \|\bar \eps^{0,1,1}\|^2
    & \leq & \sum_{i,\ \text{singular}} |\bar \eps^{0,1,1}_i|^2 + \sum_{i,\ \text{regular}} |\bar \eps^{0,1,1}_i|^2\\
    & \leq &  C (\frac{1}{h})^2 +  O(\frac{1}{h}) C  = O(\frac{1}{h^2}).
\eeno
{Altogether the} contribution {of $h\|e^1\|^2$} to the overall error {can be bounded by
\[
h\|e^1\|^2\leq C h\tau^2(\frac{1}{h^2}+h\tau^{-2\ma_4}+\tau^{-2\ma_3}+\frac1h\tau^{2-2\ma_2}  + \frac1h( \tau^{\ma_0} + h) \tau^{-2\ma_1}).
\]
}
}
\end{proof}

\section{Numerical results on two model test problems}\label{sec:numericalresults}

In this section we introduce two model test problems for diffusion with obstacle,  with source terms and analytic solutions, to better analyze the performance
of the proposed BDF schemes. A first problem mimics the American option problem with a jump in the $v_{xx}$ derivative
at a given singular position $x_s(t)$ that can be user-defined (in the numerical simulations, we will assume a $\sqrt{t}$ behavior for small times, see \eqref{eq:xs}).
The second problem allows for a bounded $v_{xxx}$ derivative with a jump at $x=x_s(t)$. These two models allow us to better check numerically
the performance of the BDF2 and BDF3 schemes, respectively. Without an analytic solution,
it is otherwise difficult to precisely compute a reference solution with very fine mesh.

\subsection{Two model test problems}
We first define two model test problems. In the case of \eqref{eq:1} we do {in general} not  know about exact
solutions.
Therefore {we construct} simple model obstacle problems
with explicit solutions (or solutions that can be easily computed with machine precision)
and also with the main features of the one-dimensional American option problem.

This is obtained by choosing an explicit function
$v=v(t,x)$ and adding a corresponding source term $f=f(t,x)$ to the original PDE {\eqref{eq:amer-pb}}, thus considering
\begin{equation}
  \min\bigg(v_t - \frac{\lambda^2}{2} x^2 v_{xx}- r x v_x + r v,\ v-{\varphi}(x)\bigg)
   =
  f(t,x).
\end{equation}

More precisely,
let $K$, $X_{max}$, $c_0$, $T$ {and $\alpha$} be given constants such that $0<K<X_{max}$, $c_0>0$, $T>0$,
{$\alpha\in(0,1]$} and such that $K-c_0\,T^{\alpha}>0$. Let ${\varphi}(x):=\max(K-x,0)$ denote the payoff function
and let $x_s$ be defined by
\be\label{eq:xs}
  x_s(t):=K(1-c_0\,t^\alpha).
\ee
{In the numerical experiments we will use $\alpha=\frac{1}{2}$, to be close to the American option case, even though the error estimate in Theorem \ref{thm:error_estim} only yields convergence for $\alpha\in(\frac12,\frac34)$ for the below Model 1, see Remark \ref{rem:thm-error-estim_applied_to_Model1}.}

We construct explicit functions $v(t,x)$ defined
for $x\in[0,X_{max}]$ and such that
\begin{itemize}
\item[$(i)$]
  $v(t,x)={\varphi}(x)= K-x$ for $x\leq x_s(t)$,
\item[$(ii)$]
  $v(t,x)>{\varphi}(x)=\max(K-x,0)$ for $x\in\,]x_s(t),X_{max}]$,
\item[$(iii)$]
  {for all $t\in(0,T]$, $v(t,\cdot)$} is at least $C^1$ on $[0,X_{max}]$,
\item[$(iv)$]
  $v(t,X_{max})=0$.
\end{itemize}
Note that requirement $(iii)$ implies $v_x(t,x_s(t))={\varphi}'(x_s(t))=-1$ for $t>0$.

\medskip

\paragraph{\bf Model 1.}
Let $v=v(t,x)$ be the function  defined by:
\be
  v(t,x):= \left\{ \barr{ll}
    {\varphi}(x)      & \mbox{for $x<x_s(t)$}\\
    {\varphi}(x_s(t))- \disp \frac{x-x_s(t)}{ 1 {+} (x-x_s(t))/C(t)} & \mbox{otherwise}
    \earr \right.
\ee
where $C(t)>0$ is a constant such that $v(t,X_{max})=0$:
$$
  {C(t) := \bigg(\frac{1}{{\varphi}(x_s(t))} - \frac{1}{X_{max}-x_s(t)}\bigg)^{-1}.}
$$
Then the requirements $(i)-(iv)$ are satisfied.

\medskip

\paragraph{\bf Model 2.}
Let $v=v(t,x)$ be the function  defined by:
\be
  v(t,x):= \left\{ \barr{ll}
    {\varphi}(x)      & \mbox{for $x<x_s(t)$}\\
    {\varphi}(x_s(t))-C(t) \atan\bigg(\disp \frac{x-x_s(t)}{C(t)}\bigg) & \mbox{otherwise}
    \earr \right.
\ee
for a given $C(t)>0$. Notice that $v(t,x)$ is a non-increasing function of the variable~$x$.
This function will satisfy requirements $(i)-(iv)$ if furthermore $C(t)$ is such that
\be\label{eq:Cconst}
  \frac{{\varphi}(x_s(t))}{C(t)} = \atan \left(\disp \frac{X_{max}-x_s(t)}{C(t)}\right).
\ee
Letting $a:=X_{max}-x_s(t)$ and $b={\varphi}(x_s(t))=K-x_s(t)$ it is clear that $0<b<a$ and therefore there exists
a unique $\mt>0$ such that $b\mt=\atan(a\mt)$.
This value can be numerically obtained by using a fixed-point method.
We then define $C(t):=1/\mt$ to obtain a solution of \eqref{eq:Cconst}. Therefore the function $v$ is in explicit
form but for the computation of the $C(t)$ function which can be computed to arbitrary precision.

\begin{rem}
For Model 2, in order to compute $v_t(t,x)$ the {derivative $\dot C(t)$} is needed.
Denoting $a=a(t)=X_{max}-x_s(t)$ and $b=b(t)={\varphi}(x_s(t))$, and $\mt=\mt(t)=1/C(t)$, by derivation of
$b\mt=\atan(a\mt)$
we obtain $\dot C/C = - \dot\mt/\mt = (q \dot b-\dot a)/(q b-a)$ where $q=1+(a/C)^2$,
with $\dot a= -\dot x_s(t)$ and $\dot b={\varphi}'(x_s(t))\dot x_s(t)$.
\end{rem}

\begin{figure}[!hbtp]
\begin{center}

\mbox{\includegraphics[width=0.56\textwidth]{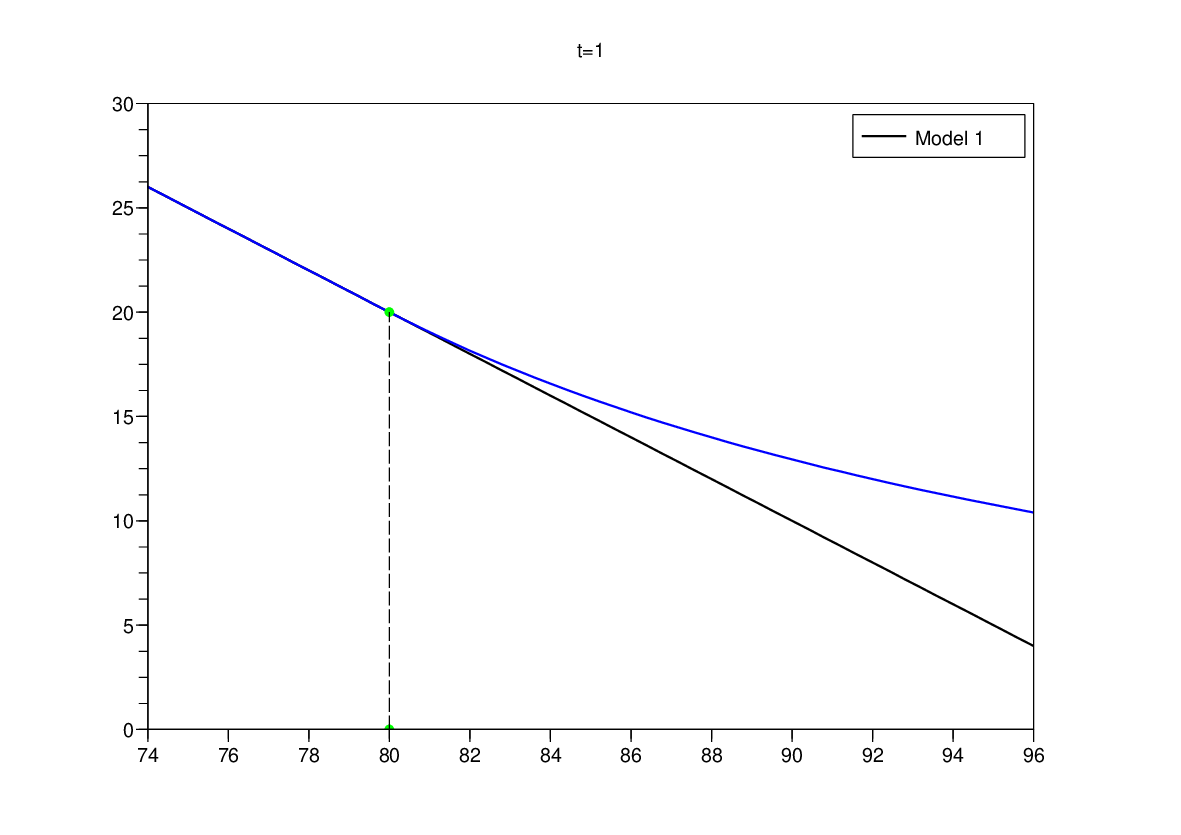}
\hspace*{-0.5cm}
\includegraphics[width=0.55\textwidth]{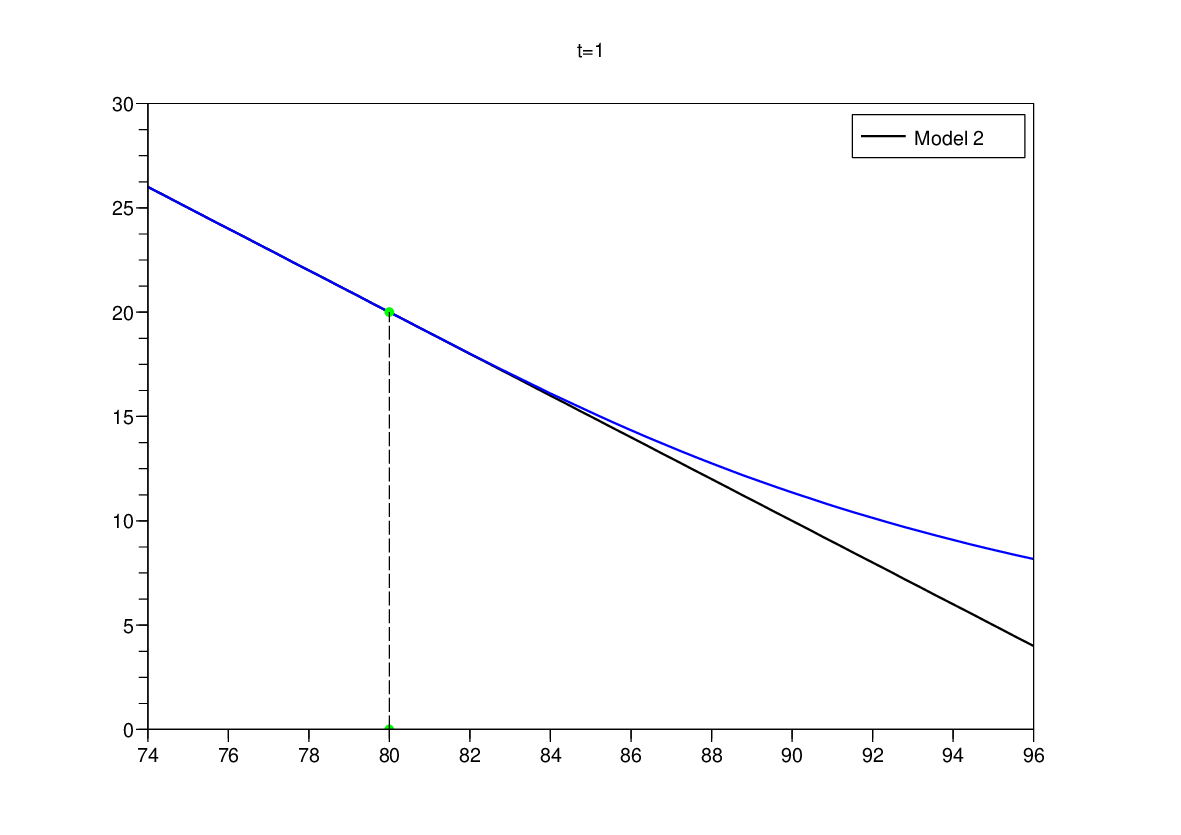}
}
\caption{Zooming around the singular point $(x_s,{\varphi}(x_s))$ for model 1 (left) and 2 (right).\label{fig:zoom}}
\end{center}
\end{figure}

\begin{figure}[!hbtp]
\includegraphics[width=0.9\textwidth,height=0.25\textheight]{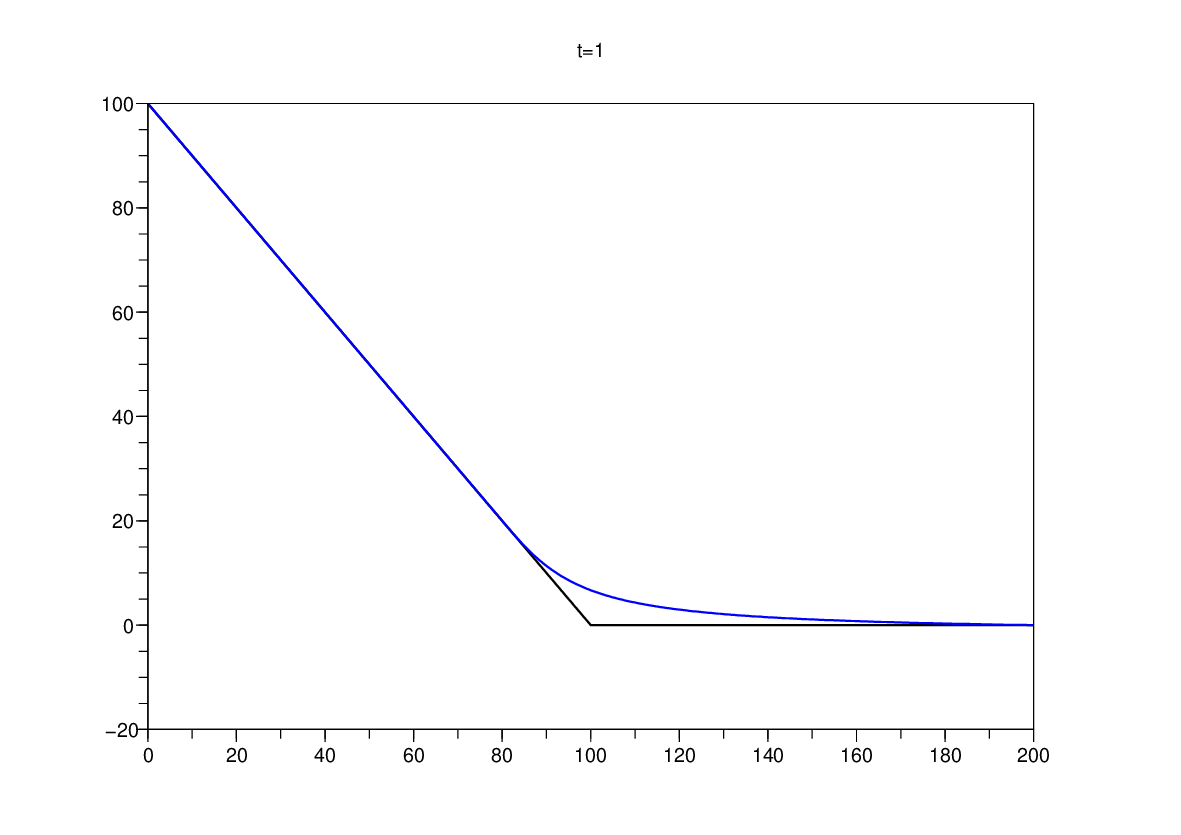}
\caption{Model 2.\label{fig:model}}
\end{figure}

\begin{rem}
The main difference between the two models is the regularity of the data near the singularity $x=x_s(t)$.
More precisely, for the first model
there is a jump in the second derivative: $v$ is of class $C^1$ and $v_{xx}$ is discontinuous.
For the second model, $v$ is of class $C^2$ and there is a jump in the third derivative $v_{3x}$.
\end{rem}

\subsection{A 4th order approximation of the spatial operator \texorpdfstring{$\cA$}{}}\label{sec:5.2}
We furthermore introduce a 4th order numerical matrix approximation of the $\cA$ operator in order to better observe the
time discretization error.

Let $D^2 u_j := \frac{-u_{j-1} + 2 u_j - u_{j+1}}{\dx^2}$.
By using Taylor expansions, if $u_j=u(x_j)$, we have
$$
  -u_{xx}(x_j) = \frac{-u_{j-1} + 2 u_j - u_{j+1}}{\dx^2}
  + \frac{1}{12\dx^2} (u_{j-2} - 4 u_{j-1} + 6 u_j - 4 u_{j+1} + u_{j+2}) + O(h^4)
$$
and
$$
   u_{x}(x_j) = \frac{u_{j+1} - u_{j-1}}{2\dx}
     + \frac{1}{12\dx} (u_{j-2} - 2 u_{j-1} + 2 u_{j+1} - u_{j+2}) + O(h^4).
$$
Therefore we have a 4-$th$ order approximation of the spatial derivatives, {and will denote by $\tilde A$ (instead of $A$)
the corresponding approximation matrix.}
At the boundaries, {for the American option problem} we use $u_{-1}=K-X_{min}-h$, $u_0=K-X_{min}$, and $u_{J+1}=u_{J+2}=0$
(the left boundary condition is consistent with fourth order because we expect that
$v(t,x)={\varphi}(x)=K-x$ near the left boundary,
also the right boundary condition is consistent with the fact that the exact solution $v(t,x)$ (obtained for the case $X_{\max}= \infty$)
decays faster than any polynomial as $x\conv \infty$). {For our model problems, we will use the known exact solution values at the boundaries.}

The BDF2 and BDF3 schemes are otherwise unchanged concerning the time discretization.
Hence for $\dt\equiv c\ \dx$ we expect to
mainly see the error of the time discretization.
In particular we aim to observe, whenever possible, second or third order behavior in time.

{
The following lemma shows that the coercivity property of Lemma~\ref{lem:4.6} extends to the 4$th$ order
approximation:

\begin{lemma}\label{lem:5xx}
Under assumption (A1),
there exist $\eta_2>0$ and $\mg_2\geq 0$ such that
for all $x\in\R^J$:
\be
  \<x,\tilde A x\> \geq \eta_2 N(x/h)^2 - \mg_2 \|x\|_2^2.
\ee
\end{lemma}
\begin{proof}
Let us focus on the matrix part that concerns the approximation of the diffusion  $-a(x) u_{xx}(x)$,
where we have removed the time dependency to simplify the notation.
The other terms coming from the drift term $b(x) u_x$ and the term $r u$ can be bounded as before.
The new matrix reads as follows:
$$\tilde A:= A + B$$
where $A$ stands for the second order approximation of the diffusion term, i.\,e.,
$$A:= \frac{1}{h^2} \pdiag(-a_j, 2a_j, -a_j) \equiv \frac{1}{h^2} \Delta A_0 $$
and
$$ B:=\frac{1}{12h^2} \pdiag( a_j,\ -4a_j,\ 6aj, -4 a_j, a_j)\equiv\frac{1}{12h^2}\Delta B_0  ,$$
where $a_j=a(x_j)$, and $\pdiag$ stands for p-band diagonal matrices (a tridiagonal matrix for $A$,
rsp.\ a pentadiagonal matrix for $B$),
and where we have also denoted $A_0:=\pdiag(-1,2,-1)$, $B_0:=\pdiag(1,-4,6,-4,1)$,
and $\Delta:=\diag(a_j)$ (a diagonal matrix with $\Delta_{jj}=a_j$).

First notice that
$$ B_0= A_0^2  +
  \begin{pmatrix}
    1 & 0 & \dots & 0 \\
    0 & 0 & \dots & 0 \\
    \vdots & \vdots & \dots & \vdots \\
    0 & \dots & 0 & 1
  \end{pmatrix}
$$
and therefore $\<B_0 x,x\>\geq \<A_0^2 x,x\>$, and, since $\Delta \geq 0$ and diagonal,
$$
  \<Bx,x\> \geq \<\frac{1}{12 h^2} \Delta A_0^2 x,x\>. 
$$

Because $x\conv a(x)$ is assumed Lipschitz continuous, we have
$$ \max_{i,j} |(\Delta A_0 - A_0 \Delta)_{ij})|_\infty  = O(h). $$
As the matrix $\Delta A_0 - A_0 \Delta$ is tridiagonal, it follows
$$ \| \Delta A_0 x - A_0 \Delta x\|\leq C h \|x\|. $$
Therefore we have also
\be
   \<\Delta A_0^2 x, x\>
     & = &  \<A_0 \Delta A_0 x, x\> + \<(\Delta A_0-A_0 \Delta )A_0 x, x\> \\
     & = &  \<\Delta A_0 x, A_0 x\> + \<(\Delta A_0-A_0 \Delta )A_0 x, x\> \\
     & \geq &  \eta_0 \|A_0 x\|^2 - C h \|A_0 x\| \|x\| \\
     & \geq &  \eta_0 \|A_0 x\|^2 - \frac{\eta_0}{2} \|A_0 x\|^2 - \frac{C^2 h^2}{2\eta_0} \|x\|^2\\
     & \geq &  \frac{\eta_0}{2} \|A_0 x\|^2 - \frac{C^2 h^2}{2\eta_0} \|x\|^2\geq - \frac{C^2 h^2}{2\eta_0} \|x\|^2
\ee
and
\be
   \frac{1}{h^2}\<\Delta A_0^2 x, x\>
     & \geq &   - \frac{C^2}{2\eta_0} \|x\|^2.
\ee
Thus we have a lower bound for $\<Bx,x\>$ of the desired type,
i.\,e.\ $\<Bx,x\> \geq - C'\|x\|^2$, and the lower bound for $\<\tilde A x,x\>$ will be of the same
type as for $\<A x,x\>$.
\end{proof}
}
{Consequently, Theorem \ref{thm:error_estim} extends to the 4th order approximation.}

\subsection{Numerical results for models 1 and 2}\label{sec:num:res1-2}

{
From now on the parameters used are $\lambda=0.3$, $r=0.1$, $K=100$.
The singularity motion is defined by $x_s(t)=K(1-c_0 \sqrt{t})$ with $c_0=0.2$. 

The errors in $L^2$, $L^1$ and $L^\infty$ norms are computed at time $t_N=T$ {using \eqref{discreteLperrordef}.}
For Model 1, the numerical domain is defined by $\mO=(75,275)$ (i.\,e., $X_{\min}=75$ and $X_{\max}=275$), and $T=1$.
Therefore the singularity at $T=1$ is located at $x_s(T)=80$.
Numerical results for Model 1 with the CN and BDF2 schemes {and second order spatial discretization (results for fourth order approximation in space are similar)} are given in
Table~\ref{tab:MODEL1-CN} and
Table~\ref{tab:MODEL1-BDF2}.

For Model 1, with bounded $v_{xx}$ derivative,
we observe that the order of the CN scheme, {when $N=J{+1}$, is two in the $L^1$ and $L^2$ norm (and around $1.6$ in the $L^\infty$ norm)}
and goes down to order one {in the $L^\infty$ norm} when $N={(}J{+1)}/10$ (i.\,e.\ the mesh ratio $\dt/h$ is large).
On the contrary, the BDF2 scheme keeps roughly an error of order two for different mesh ratios and for all norms.

\begin{rem}\label{rem:firststep}
For models 1 and 2, we have observed that for the first step of BDF2, using the BDF1 scheme
  (i.\,e., the {IE} scheme) instead of CN {yields nearly unchanged results} {(this is different to the American option problem,
  see Remark~\ref{rem:technical_details})}.
\end{rem}

\begin{rem}\label{rem:BDF3-MODEL1}
We have also tested the BDF3 scheme on Model 1, and the numerical results are {(both for second and fourth order approximation in space)} very similar to the BDF2 scheme,
and in particular the numerical order is not greater than two.
This comes from the fact that the solution has a bounded second order derivative, with a jump.
\end{rem}
}

\begin{table}[!hbtp]
\begin{center}
\begin{tabular}{|cc|cc|cc|cc|c@{\hspace{3pt}}|}
\hline
\multicolumn{2}{|c|}{Mesh}
   & \multicolumn{2}{|c|}{Error $L^1$} & \multicolumn{2}{|c|}{Error $L^2$} & \multicolumn{2}{|c|}{Error $L^\infty$}
   &  time(s)
   \\
   \hline
  $J{+1}$ & $N$   &  error   &  order &    error & order  & error    & order   &         \\
\hline \hline
    80 &     80 &{1.09e+00} &{ 1.79 }& {1.67e-01}& {1.75 }& {3.89e-02 }& {1.64 }& { 0.04}\\
   160 &    160 &{2.85e-01} &{ 1.94 }& {4.53e-02}& {1.88 }& {1.14e-02 }& {1.77 }& { 0.08}\\
   320 &    320 &{8.41e-02} &{ 1.76 }& {1.36e-02}& {1.73 }& {3.63e-03 }& {1.65 }& { 0.22}\\
   640 &    640 &{2.12e-02} &{ 1.99 }& {3.51e-03}& {1.95 }& {9.94e-04 }& {1.87 }& { 0.53}\\
  1280 &   1280 &{5.68e-03} &{ 1.90 }& {9.46e-04}& {1.89 }& {2.76e-04 }& {1.85 }& { 1.40}\\
  2560 &   2560 &{1.41e-03} &{ 2.00 }& {2.40e-04}& {1.98 }& {7.36e-05 }& {1.91 }& { 4.52}\\
  5120 &   5120 &{3.65e-04} &{ 1.96 }& {6.23e-05}& {1.95 }& {1.95e-05 }& {1.92 }& {17.04}\\
 10240 &  10240 &{9.08e-05} &{ 2.01 }& {1.57e-05}& {1.99 }& {7.00e-06 }& {1.48 }& {64.24}\\
\hline\hline
    80 &      8 &{5.07e+00} &{ 1.61 }& {1.02e+00}& {1.39 }& {4.35e-01 }& {0.69 }& { 0.00}\\
   160 &     16 &{8.93e-01} &{ 2.51 }& {1.83e-01}& {2.48 }& {9.27e-02 }& {2.23 }& { 0.01}\\
   320 &     32 &{3.00e-01} &{ 1.57 }& {6.83e-02}& {1.42 }& {4.34e-02 }& {1.09 }& { 0.03}\\
   640 &     64 &{6.12e-02} &{ 2.29 }& {1.79e-02}& {1.93 }& {2.10e-02 }& {1.05 }& { 0.06}\\
  1280 &    128 &{1.79e-02} &{ 1.77 }& {5.86e-03}& {1.61 }& {1.01e-02 }& {1.06 }& { 0.18}\\
  2560 &    256 &{4.17e-03} &{ 2.10 }& {1.86e-03}& {1.66 }& {4.94e-03 }& {1.03 }& { 0.58}\\
  5120 &    512 &{1.11e-03} &{ 1.91 }& {6.21e-04}& {1.58 }& {2.38e-03 }& {1.06 }& { 2.25}\\
 10240 &   1024 &{2.73e-04} &{ 2.02 }& {2.15e-04}& {1.53 }& {1.19e-03 }& {1.00 }& { 8.23}\\
\hline
\end{tabular}
\end{center}
\caption{
\label{tab:MODEL1-CN}
(Model 1)
CN scheme for \eqref{eq:1} (using {2nd} order spatial approximation),
with different mesh ratio{s} $N=J{+1}$, $N={(}J{+1)}/10$.}
\end{table}

\begin{table}[!hbtp]
\begin{center}
\begin{tabular}{|cc|cc|cc|cc|c|}
\hline
\multicolumn{2}{|c|}{Mesh}
   & \multicolumn{2}{|c|}{Error $L^1$} & \multicolumn{2}{|c|}{Error $L^2$} & \multicolumn{2}{|c|}{Error $L^\infty$}
   &  time(s)
   \\
   \hline
  $J{+1}$ & $N$   &  error   &  order &    error & order  & error    & order   &         \\
\hline \hline
    80 &     80 &{1.27e+00} &{ 1.87 }& {2.03e-01}& {1.77 }& {7.15e-02 }& {1.39 }& { 0.02}\\
   160 &    160 &{3.59e-01} &{ 1.82 }& {5.89e-02}& {1.78 }& {2.50e-02 }& {1.51 }& { 0.05}\\
   320 &    320 &{9.18e-02} &{ 1.97 }& {1.55e-02}& {1.93 }& {8.17e-03 }& {1.61 }& { 0.12}\\
   640 &    640 &{2.39e-02} &{ 1.94 }& {4.06e-03}& {1.93 }& {2.52e-03 }& {1.70 }& { 0.32}\\
  1280 &   1280 &{5.97e-03} &{ 2.00 }& {1.03e-03}& {1.98 }& {7.35e-04 }& {1.78 }& { 0.96}\\
  2560 &   2560 &{1.51e-03} &{ 1.98 }& {2.60e-04}& {1.98 }& {2.06e-04 }& {1.84 }& { 3.25}\\
  5120 &   5120 &{3.76e-04} &{ 2.01 }& {6.49e-05}& {2.00 }& {5.58e-05 }& {1.88 }& {11.56}\\
 10240 &  10240 &{9.43e-05} &{ 1.99 }& {1.63e-05}& {2.00 }& {1.48e-05 }& {1.91 }& {43.10}\\
\hline\hline
    80 &      8 &{1.83e+00} &{ 2.57 }& {4.24e-01}& {2.10 }& {1.93e-01 }& {1.76 }& { 0.00}\\
   160 &     16 &{3.81e-01} &{ 2.26 }& {9.19e-02}& {2.21 }& {5.32e-02 }& {1.86 }& { 0.01}\\
   320 &     32 &{7.93e-02} &{ 2.27 }& {2.03e-02}& {2.18 }& {1.44e-02 }& {1.89 }& { 0.01}\\
   640 &     64 &{1.96e-02} &{ 2.02 }& {4.71e-03}& {2.11 }& {3.81e-03 }& {1.92 }& { 0.04}\\
  1280 &    128 &{4.47e-03} &{ 2.13 }& {1.05e-03}& {2.17 }& {9.90e-04 }& {1.94 }& { 0.12}\\
  2560 &    256 &{1.13e-03} &{ 1.98 }& {2.48e-04}& {2.08 }& {2.55e-04 }& {1.96 }& { 0.40}\\
  5120 &    512 &{2.72e-04} &{ 2.06 }& {5.74e-05}& {2.11 }& {6.50e-05 }& {1.97 }& { 1.51}\\
 10240 &   1024 &{6.99e-05} &{ 1.96 }& {1.41e-05}& {2.02 }& {1.65e-05 }& {1.98 }& { 5.15}\\
\hline
\end{tabular}
\end{center}
\caption{
\label{tab:MODEL1-BDF2}
(Model 1)
BDF2 scheme (using {2nd} order spatial approximation), for different mesh ratios.}
\end{table}


{
Then we focus on numerical results for Model 2. We have tested again the CN, BDF2 and BDF3 schemes.
In that case we consider the problem with $\mO=(50,450)$ and $T=0.5$, the other parameters being as in Model~1.

Results for CN and BDF3 schemes are given
in Tables~\ref{tab:MODEL2-CN} and ~\ref{tab:MODEL2-BDF3} respectively.
For this model, by construction, we recall that the exact solution has bounded third order spacial derivatives.
The CN scheme gives good results when $N=J{+1}$ (second order convergence), but goes back to first order convergence when $N={(}J{+1)}/10$
(in the $L^\infty$ norm). The results for the BDF2 scheme, which are not shown, demonstrate second order convergence but unconditionally on the mesh parameters.
On the other hand the BDF3 scheme shows at least
third order convergence for the $L^\infty$ norm, as well for both ratios of the mesh parameters.

In conclusion, for the type of obstacle problems studied here,
we advise using the BDF2 scheme instead of the CN scheme because it keeps its expected numerical order
unconditionally on the mesh parameters.
}

\begin{table}[!hbtp]
\begin{center}
\begin{tabular}{|cc|cc|cc|cc|c|}
\hline
\multicolumn{2}{|c|}{Mesh}
   & \multicolumn{2}{|c|}{Error $L^1$} & \multicolumn{2}{|c|}{Error $L^2$} & \multicolumn{2}{|c|}{Error $L^\infty$}
   &  time(s)
   \\
   \hline
  $J{+1}$ & $N$   &  error   &  order &    error & order  & error    & order   &         \\
\hline\hline
     80 &    80  & 8.04{e}-01 &  3.14  & 2.13{e}-01 &   2.56 & 8.32{e}-02 &  1.67 & { 0.16}\\ 
    160 &   160  & 1.03{e}-01 &  2.96  & 2.14{e}-02 &   3.31 & 7.25{e}-03 &  3.52 & { 0.26}\\ 
    320 &   320  & 8.46{e}-03 &  3.61  & 1.64{e}-03 &   3.71 & 4.63{e}-04 &  3.97 & { 0.74}\\ 
    640 &   640  & 3.11{e}-04 &  4.77  & 5.80{e}-05 &   4.82 & 1.31{e}-05 &  5.14 & { 1.33}\\ 
   1280 &  1280  & 6.26{e}-06 &  5.64  & 1.29{e}-06 &   5.50 & 5.35{e}-07 &  4.61 & { 3.23}\\ 
   2560 &  2560  & 5.81{e}-07 &  3.43  & 1.22{e}-07 &   3.40 & 4.06{e}-08 &  3.72 & { 7.99}\\ 
   5120 &  5120  & 1.43{e}-07 &  2.02  & 2.96{e}-08 &   2.04 & 8.70{e}-09 &  2.22 & {22.41}\\ 
  10240 & 10240  & 3.5{7}e-08&  2.01  & 7.40{e}-09 &   2.00 & 2.17{e}-09 &  2.00 & {64.65}\\ 
\hline\hline
     80 &     8  & 8.07{e}-01 &  3.33  & 2.20{e}-01 &   2.70 & 8.83{e}-02 &  1.76 &{ 0.01}\\ 
    160 &    16  & 1.44{e}-01 &  2.49  & 3.37{e}-02 &   2.71 & 1.40{e}-02 &  2.66 &{ 0.02}\\ 
    320 &    32  & 1.48{e}-02 &  3.28  & 3.34{e}-03 &   3.34 & 1.36{e}-03 &  3.37 &{ 0.06}\\ 
    640 &    64  & 1.04{e}-03 &  3.83  & 3.62{e}-04 &   3.21 & 2.36{e}-04 &  2.52 &{ 0.14}\\ 
   1280 &   128  & 2.50{e}-04 &  2.06  & 8.43{e}-05 &   2.10 & 7.91{e}-05 &  1.58 &{ 0.28}\\ 
   2560 &   256  & 7.81{e}-05 &  1.68  & 3.30{e}-05 &   1.36 & 4.02{e}-05 &  0.98 &{ 0.71}\\ 
   5120 &   512  & 2.02{e}-05 &  1.95  & 1.09{e}-05 &   1.60 & 1.97{e}-05 &  1.03 &{ 2.05}\\ 
  10240 &  1024  & 5.27{e}-06 &  1.94  & 3.80{e}-06 &   1.52 & 1.01{e}-05 &  0.97 &{ 6.52}\\ 
\hline
\end{tabular}
\end{center}
\caption{
\label{tab:MODEL2-CN}
(Model 2)
CN scheme for \eqref{eq:1} (using 4th order spatial approximation).}
\end{table}

\begin{table}[!hbtp]
\begin{center}
\begin{tabular}{|cc|cc|cc|cc|c|}
\hline
\multicolumn{2}{|c|}{Mesh}
   & \multicolumn{2}{|c|}{Error $L^1$} & \multicolumn{2}{|c|}{Error $L^2$} & \multicolumn{2}{|c|}{Error $L^\infty$}
   &  time(s)
   \\
   \hline
  $J{+1}$ & $N$   &  error   &  order &    error & order  & error    & order   &         \\
\hline\hline
     80 &    80  & 8.07{e}-01 &  3.13  & 2.13{e}-01 &   2.55 & 8.33{e}-02 &  1.67 &{ 0.09}\\
    160 &   160  & 1.01{e}-01 &  2.99  & 2.10{e}-02 &   3.34 & 7.12{e}-03 &  3.55 &{ 0.18}\\
    320 &   320  & 8.30{e}-03 &  3.61  & 1.60{e}-03 &   3.71 & 4.52{e}-04 &  3.98 &{ 0.40}\\
    640 &   640  & 3.04{e}-04 &  4.77  & 5.67{e}-05 &   4.82 & 1.29{e}-05 &  5.13 &{ 0.93}\\
   1280 &  1280  & 7.27{e}-06 &  5.38  & 1.40{e}-06 &   5.34 & 4.86{e}-07 &  4.73 &{ 2.17}\\
   2560 &  2560  & 1.34{e}-07 &  5.77  & 4.43{e}-08 &   4.98 & 2.46{e}-08 &  4.30 &{ 5.53}\\
   5120 &  5120  & 1.1{3}e-08&3.5{7}& 2.8{0}e-09& 3.9{9}& 1.41e-09& 4.1{2}& {18.26}\\
  10240 &  10240 &{1.07e-09} &{3.40}& {2.13e-10}& {3.72}& {7.88e-11}& 4.{16} & {49.88}\\
\hline\hline
     80 &     8  & 1.02E+00 &  2.74  & 2.41{e}-01 &   2.38 & 9.12{e}-02 &  1.84 & { 0.01}\\
    160 &    16  & 6.54{e}-02 &  3.96  & 1.69{e}-02 &   3.84 & 6.99{e}-03 &  3.71 &{0.02}\\
    320 &    32  & 9.64{e}-03 &  2.76  & 1.86{e}-03 &   3.18 & 5.28{e}-04 &  3.73 &{0.04}\\
    640 &    64  & 3.53{e}-04 &  4.77  & 6.58{e}-05 &   4.82 & 1.54{e}-05 &  5.10 &{0.10}\\
   1280 &   128  & 1.46{e}-05 &  4.60  & 2.69{e}-06 &   4.61 & 6.31{e}-07 &  4.61 &{0.22}\\
   2560 &   256  & 8.46{e}-07 &  4.11  & 1.57{e}-07 &   4.10 & 3.88{e}-08 &  4.02 &{0.54}\\
   5120 &   512  & 8.6{2}{e}-08 &  3.29  & 1.64{e}-08 &   3.26 & 4.2{5}{e}-09 &  3.19 & { 1.53}\\
  10240 &  1024  &{1.01e-08}&  3.1{0}& 1.{9}6{e}-09 &   3.{06}& {5.21}{e}-10 &  3.{03}&  {5.00}\\
\hline
\end{tabular}
\end{center}
\caption{
\label{tab:MODEL2-BDF3}
(Model 2)
BDF3 scheme for \eqref{eq:1} (using 4th order spatial approximation).}
\end{table}

\appendix

\section{An HJB equation for obstacle problems}
\label{app:proof-of-HJB}
This appendix is devoted to a sketch of proof for the equivalence between PDE \eqref{eq:1} and PDE \eqref{eq:2} {in}
case the coefficients are not time dependent.

In order to simplify the presentation we assume that $f\equiv 0$ and $\mO\equiv \R$.
We consider the problem \eqref{eq:1} after a change of variable $t\conv T-t$:
\begin{subequations}\label{eq:4}
\be
  & & \min (-v_t + \cA v,\ v -\varphi(x))=0, \quad t\in(0,T), \quad x \in \mO,\\
  & & v(T,x)=\varphi(x), \quad x\in \mO
\ee
\end{subequations}
We aim to prove that $v$ is also a viscosity solution of \eqref{eq:2}.
{In the following, we will first prove that $(i)$ $-v_t + \cA v\geq 0$, then $(ii)$ that $-v_t\geq 0$, then $(iii)$ that $\min(-v_t + \cA v, -v_t) =0$
and will $(iv)$ conclude by a uniqueness argument.}

$(i)$ By uniqueness of the continuous solutions of \eqref{eq:1},
$v$ is also given by the expectation formula
$$
  v(t,x) = \sup_{\tau\in \cT_{[t,T]}} \E ( e^{-\int_t^\tau r\ds } \varphi(X^{t,x}_\tau) | \cF_t ).
$$
(see for instance \cite{lam-lap-2008})
where we have considered a probability space $(\mO,\cF,\P)$, a filtration $(\cF_t)_{t\geq 0}$,
$\cT_{[t,T]}$ is the set of stopping times taking values a.s.\ in $[t,T]$,
$X_\tau:=X^{t,x}_\tau$ is the strong solution of the stochastic  differential equation (SDE):
$$ \dop X_s = b(X_s) \ds +  \ms(X_s) \dop W_s, \quad s\geq t,$$
with $X_t =x$,
$W_s$ denotes an $\cF_t$-adapted Brownian motion on $\R$,
and the \enquote{sup} is an essential supremum over $\cT_{[t,T]}$.
First one can use the semi-Martingale property
$$
  v(t,x)\leq \E ( e^{-r h}\, v(t+h, X^{t,x}_{t+h}) | \cF_t),
$$
in order to deduce (in the viscosity sense), that
$$ - v_t  + \cA v  \geq 0.$$

$(ii)$ Then we aim to show that $v(t,x) \geq v(t+h,x)$, for any $h>0$. This will imply
$-v_t \geq 0$ (in the viscosity sense).
By definition,
\be
  v(t+h,x) & = &  \sup_{\tau\in \cT_{[t+h,T]}} \E ( e^{-\int_{t+h}^\tau r\, \ds} \varphi(X^{t+h,x}_\tau) | \cF_{t+h} )\\
           & = &  \sup_{\tau\in \cT_{[t,T-h]}} \E ( e^{-\int_{t+h}^{\tau+h} r\, \ds} \varphi(X^{t+h,x}_{\tau+h}) | \cF_{t+h} )\\
           & = &  \sup_{\tau\in \cT_{[t,T-h]}} \E ( e^{-\int_{t}^\tau r\, \ds} \varphi(X^{t,x}_{\tau}) | \cF_{t} ).
\ee
We have used the fact that the process $X^{t,x}$ satisfies an SDE with no time dependency in the coefficients,
and also, since $\tau\in \cT_{[t,T-h]}$, $X_{\tau+h}$ a.s.\  stops before time $T$,
the fact that  $\E(X^{t+h,x}_{\tau+h}|\cF_{t+h}) =\E( X^{t,x}_{\tau} | \cF_t)$ -
which corresponds to an averaging during a period
of time $T-(t+h)$.
Then, in particular,
\be
  v(t+h,x) & \leq &  \sup_{\tau\in \cT_{[t,T]}} \E ( e^{-\int_{t}^\tau r\, \ds} \varphi(X^{t,x}_{\tau}) | \cF_{t} )
    = v(t,x).
\ee

At this point we therefore have shown that
$$ \min(-v_t +\cA v, -v_t) \geq 0.$$

$(iii)$ Let us assume that $-v_t(t,x)>0$ (in the viscosity sense), and $t<T$.
It implies that $v(t,x)>v(t+h,x)$ for all $h>0$ small enough.
Because $v(t,x)>v(t+h,x)\geq \varphi(x)$, we have $v(t,x)> \varphi(x)$.
The following dynamic programming principle holds:
$$
  v(t,x)= \E ( e^{-\int_{t}^{\tau^*_{t,x}} r\,\ds} \varphi(X^{t,x}_{\tau^*_{t,x}}) | \cF_t )
        =  \E ( e^{-\int_{t}^{\tau^*_{t,x}} r\,\ds}  v(\tau^*_{t,x} , X^{t,x}_{\tau^*_{t,x}}) | \cF_t )
$$
where $\tau^{*}_{t,x}$ is the optimal stopping time for the obstacle problem, defined by
$$
  \tau^{*}_{t,x} = \inf \bigg\{ \mt \geq t,\ v(\mt,X^{t,x}_\mt) = \varphi(X^{t,x}_\mt) \bigg\}.
$$
It can be shown that $\tau^{*}_{t,x} >t$ a.s.\ (since $v(t,x)>\varphi(x)$, these functions being continuous).
{Let us show that $-v_t + \cA v = 0$ at $(t,x)$ in the viscosity sense}.
By using Ito's formula between $t$ and $\tau^*_{t,x}$, and from the dynamic programming principle,
we deduce that
\beno
   0 =  \E \bigg(\int_t^{\tau^*_{t,x}} e^{-\int_t^{\mt} r\,\ds} (v_t - \cA v)_{(\mt,X^{t,x}_\mt)}\dop\mt\ |\ \cF_t \bigg).
\eeno
We already have proved that $v_t - \cA v \leq 0$ a.s., so we deduce that
$(v_t - \cA v)(\mt,x) = 0$ a.e.\ for $\mt\in(t,\tau^*_{t,x})$. For some random parameter $w$ we have $t^*:=\tau^*_{t,x}(w)>t$,
from which it is deduced that $(v_t - \cA v)(t,x) = 0$.
Therefore we have proved in this case that $\min(-v_t(t,x) + \cA v(t,x),\ -v_t)=0$.

$(iv)$ Conversely, we can use a uniqueness argument for the solutions of \eqref{eq:2} in order to conclude the equivalence between
\eqref{eq:1} and \eqref{eq:2}.

\begin{rem}\label{rem:equivalence-with-nonzero-f}
{\em
In the same way, it can be proved that the following PDE with source term and $x$-dependent coefficients in the operator $\cA$:
\begin{subequations}\label{eq:PDEsource}
\be
  & & \min (-u_t + \cA u, u -\varphi(x))=f(x), \quad t\in(0,T), \quad x \in \mO,\\
  & & u(T,x)=\varphi(x)+f(x), \quad x\in \mO,
\ee
\end{subequations}
is equivalent to the following Hamilton-Jacobi-Bellman equation
\begin{subequations}\label{eq:PDEsource.ut}
\be
  & & -u_t + \min(\cA u,\ 0) = f(x), \quad t\in(0,T), \quad x \in \mO,\\
  & & u(T,x)=\varphi(x)+f(x), \quad x\in \mO.
\ee
\end{subequations}
Problem \eqref{eq:PDEsource} is associated with the following stopping time problem
\be\label{eq:stop-time-pb-phi-f}
  u(t,x) = \sup_{\tau\in \cT_{[t,T]}} \E \bigg( e^{-\int_t^\tau r \ds } (\varphi(X^{t,x}_\tau)+f(X^{t,x}_\tau))
     + \int_t^\tau e^{-\int_t^\mt r \ds} f(X^{t,x}_\mt)\dop\mt   | \cF_t \bigg).
\ee
}
{The Hamilton-Jacobi equation~\eqref{eq:PDEsource.ut} (or \eqref{eq:2})
admits also a representation formula corresponding to a stochastic optimal control problem with controlled diffusion, drift and rate term
$(\mt\ms(t,x),\,\mt b(t,x),\,\mt r(t))$ with $\mt\in[0,1]$,
see for instance~\cite{Lions_1983a}.}
\end{rem}

\def\cprime{$'$}


\begin{thebibliography}{10}

\bibitem{ach-pir-2005}
Yves Achdou and Olivier Pironneau.
\newblock {\em Computational methods for option pricing}, volume~30 of {\em
  Frontiers in Applied Mathematics}.
\newblock Society for Industrial and Applied Mathematics (SIAM), Philadelphia,
  PA, 2005.

\bibitem{bar-bur-rom-sam-1993}
Guy Barles, Julien Burdeau, Marc Romano, and Nicolas Sams{\oe}n.
\newblock Estimation de la fronti\`ere libre des options am\'ericaines au
  voisinage de l'\'ech\'eance.
\newblock {\em C. R. Acad. Sci. Paris S\'er. I Math.}, 316(2):171--174, 1993.

\bibitem{ber-eym-2006}
Julien Berton and Robert Eymard.
\newblock Finite volume methods for the valuation of {A}merican options.
\newblock {\em M2AN Math. Model. Numer. Anal.}, 40(2):311--330, 2006.

\bibitem{bla-dol-mon-06}
Adrien Blanchet, Jean Dolbeault, and R\'egis Monneau.
\newblock On the continuity of the time derivative of the solution to the
  parabolic obstacle problem with variable coefficients.
\newblock {\em J. Math. Pures Appl. (9)}, 85(3):371--414, 2006.

\bibitem{BokaDebra2020}
Olivier Bokanowski and Kristian Debrabant.
\newblock {Matlab Code: Backward differentiation formula finite difference
  schemes for diffusion equations with an obstacle term}, March 2020.
\newblock https://doi.org/10.5281/zenodo.3696678.

\bibitem{bok-mar-zid-2009}
Olivier Bokanowski, Stefania Maroso, and Hasnaa Zidani.
\newblock Some convergence results for {H}oward's algorithm.
\newblock {\em SIAM J. Numer. Anal.}, 47(4):3001--3026, 2009.

\bibitem{Bokanowski_Picarelli_Reisinger_2018}
Olivier Bokanowski, Athena Picarelli, and Christoph Reisinger.
\newblock Stability and convergence of second order backward differentiation
  schemes for parabolic {H}amilton-{J}acobi-{B}ellman equations.
\newblock Preprint, 2018.

\bibitem{cia-82}
Philippe~G. Ciarlet.
\newblock {\em Introduction \`a l'analyse num\'erique matricielle et \`a
  l'optimisation}.
\newblock Collection Math\'ematiques Appliqu\'ees pour la Ma\^\i trise.
  [Collection of Applied Mathematics for the Master's Degree]. Masson, Paris,
  1982.

\bibitem{CraIshLio92}
Michael~Grain Crandall, Hitoshi Ishii, and Pierre-Louis Lions.
\newblock User's guide to viscosity solutions of second order partial
  differential equations.
\newblock {\em Bull. Amer. Math. Soc. (N.S.)}, 27(1):1--67, 1992.

\bibitem{crouzeix84and}
Michel Crouzeix and Alain~L. Mignot.
\newblock {\em Analyse num\'erique des \'equations diff\'erentielles}.
\newblock Collection Math\'ematiques Appliqu\'ees pour la Ma\^\i trise.
  [Collection of Applied Mathematics for the Master's Degree]. Masson, Paris,
  1984.

\bibitem{curtiss52ios}
Charles~Francis Curtiss and Joseph~Oakland Hirschfelder.
\newblock Integration of stiff equations.
\newblock {\em Proc. Nat. Acad. Sci. U. S. A.}, 38:235--243, 1952.

\bibitem{dew-how-rup-wil-1993}
Jeffrey~N. Dewynne, Sam~D. Howison, I.~Rupf, and Paul Wilmott.
\newblock Some mathematical results in the pricing of {A}merican options.
\newblock {\em European J. Appl. Math.}, 4(4):381--398, 1993.

\bibitem{emm-05}
Etienne Emmrich.
\newblock Stability and error of the variable two-step {BDF} for semilinear
  parabolic problems.
\newblock {\em J. Appl. Math. \& Computing}, 19(1-2):33--55, 2005.

\bibitem{Forsyth_Vetzal_2002}
Peter~A. Forsyth and Kenneth~R. Vetzal.
\newblock Quadratic convergence for valuing {A}merican options using a penalty
  method.
\newblock {\em SIAM J. Sci. Comput.}, 23(6):2095--2122, 2002.

\bibitem{friedman-88}
Avner Friedman.
\newblock {\em Variational principles and free-boundary problems}.
\newblock Robert E. Krieger Publishing Co., Inc., Malabar, FL, second edition,
  1988.

\bibitem{gear71niv}
C.~William Gear.
\newblock {\em Numerical initial value problems in ordinary differential
  equations}.
\newblock Prentice-Hall, Inc., Englewood Cliffs, N.J., 1971.

\bibitem{Hairer_Wanner_96}
Ernst Hairer and Gerhard Wanner.
\newblock {\em Solving ordinary differential equations. {II}}.
\newblock Springer-Verlag, Berlin, second edition, 1996.

\bibitem{hai-wan-10}
Ernst Hairer and Gerhard Wanner.
\newblock Linear multistep method.
\newblock {\em Scholarpedia}, 5(4):4591, 2010.

\bibitem{hin-ito-kun-2003}
Michael Hinterm{\"u}ller, Kazufumi Ito, and Karl Kunisch.
\newblock The primal-dual active set strategy as a semismooth {N}ewton method.
\newblock {\em SIAM J. Optim.}, 13(3):865--888 (2003), 2002.

\bibitem{jai-lam-lap-90}
Patrick Jaillet, Damien Lamberton, and Bernard Lapeyre.
\newblock Variational inequalities and the pricing of {A}merican options.
\newblock {\em Acta Appl. Math.}, 21(3):263--289, 1990.

\bibitem{jak-2003}
Espen~R. Jakobsen.
\newblock On the rate of convergence of approximation schemes for {B}ellman
  equations associated with optimal stopping time problems.
\newblock {\em Mathematical Models and Methods in Applied Sciences},
  13(05):613--644, 2003.

\bibitem{lam-lap-2008}
Damien Lamberton and Bernard Lapeyre.
\newblock {\em Introduction to stochastic calculus applied to finance}.
\newblock Chapman \& Hall/CRC Financial Mathematics Series. Chapman \&
  Hall/CRC, Boca Raton, FL, second edition, 2008.

\bibitem{lefloch14tbf}
Fabien Le~Floc'h.
\newblock {TR-BDF2} for fast stable {A}merican option pricing.
\newblock {\em Journal of Computational Finance}, 17(3):31--561, 2014.

\bibitem{Lions_1983a}
Pierre-Louis Lions.
\newblock Optimal control of diffusion processes and
  {H}amilton-{J}acobi-{B}ellman equations. {I}. {T}he dynamic programming
  principle and applications.
\newblock {\em Comm. Partial Differential Equations}, 8(10):1101--1174, 1983.

\bibitem{Lions_1983b}
Pierre-Louis Lions.
\newblock Optimal control of diffusion processes and
  {H}amilton-{J}acobi-{B}ellman equations. {II}. {V}iscosity solutions and
  uniqueness.
\newblock {\em Comm. Partial Differential Equations}, 8(11):1229--1276, 1983.

\bibitem{martini:inria-00072718}
Claude Martini.
\newblock {A}merican option prices as unique viscosity solutions to degenerated
  {H}amilton-{J}acobi-{B}ellman equations.
\newblock Research Report RR-3934, {INRIA}, 2000.

\bibitem{oos-2003}
Cornelis~W. Oosterlee.
\newblock On multigrid for linear complementarity problems with application to
  {A}merican-style options.
\newblock {\em Electron. Trans. Numer. Anal.}, 15:165--185 (electronic), 2003.
\newblock Tenth Copper Mountain Conference on Multigrid Methods (Copper
  Mountain, CO, 2001).

\bibitem{oos-gas-fri-2003}
Cornelis~W. Oosterlee, Francisco~José Gaspar, and J.~C. Frisch.
\newblock W{ENO} and blended {BDF} discretizations for option pricing problems.
\newblock In {\em Numerical mathematics and advanced applications}, pages
  419--428. Springer Italia, Milan, 2003.

\bibitem{Pham_1998}
Huy\^{e}n Pham.
\newblock Optimal stopping of controlled jump diffusion processes: a viscosity
  solution approach.
\newblock {\em J. Math. Systems Estim. Control}, 8(1):27 pp.\, 1998.

\bibitem{Reisinger_Whitley_2014}
Christoph Reisinger and Alan Whitley.
\newblock The impact of a natural time change on the convergence of the
  {C}rank-{N}icolson scheme.
\newblock {\em IMA J. Numer. Anal.}, 34(3):1156--1192, 2014.

\bibitem{Reisinger_Zhang_Preprint2018}
Christoph Reisinger and Yufei Zhang.
\newblock A penalty scheme for monotone systems with interconnected obstacles:
  convergence and error estimates.
\newblock Preprint, 2018.

\bibitem{serdjukova67uso}
S.~I. Serdjukova.
\newblock Uniform stability of a six-point scheme of higher order accuracy for
  the heat equation.
\newblock {\em \v Z. Vy\v cisl. Mat. i Mat. Fiz.}, 7(1):214--218, 1967.

\bibitem{seydel-12}
R{\"u}diger~U. Seydel.
\newblock {\em Tools for computational finance}.
\newblock Universitext. Springer London, fifth edition, 2012.

\bibitem{windcliff01soa}
Heath Windcliff, Peter~A. Forsyth, and Kenneth~R. Vetzal.
\newblock Shout options: a framework for pricing contracts which can be
  modified by the investor.
\newblock {\em J. Comput. Appl. Math.}, 134(1-2):213--241, 2001.

\bibitem{Witte_Reisinger_2012}
Jan~Hendrik Witte and Christoph Reisinger.
\newblock Penalty methods for the solution of discrete {HJB}
  equations---continuous control and obstacle problems.
\newblock {\em SIAM J. Numer. Anal.}, 50(2):595--625, 2012.

\end{thebibliography}
\end{document}